\documentclass[11pt]{article}

\usepackage{amsmath, amsthm, amsfonts}
\usepackage{listings}
\usepackage{amsfonts}
\usepackage{amssymb}
\usepackage{float}
\usepackage{anysize}
\usepackage{graphicx}
\usepackage{mathrsfs} 
\usepackage{hyperref}
\usepackage{cite}
\usepackage{color}
\usepackage{subcaption}
\usepackage{color}
\usepackage{cancel}
\usepackage{bbm}
\usepackage{tikz}
\usepackage{lineno} 
           
\usepackage{fullpage} 
\newtheorem{thm}{Theorem}[section]

\newtheorem{lem}[thm]{Lemma}

\theoremstyle{definition}

\theoremstyle{remark}
\newtheorem{rem}{Remark}

\usepackage{authblk} 

\title{Parameter estimation for random sampled Regression Model with Long Memory Noise}

\author[1]{H\'ector Araya }
\author[2]{Natalia Bahamonde}
\author[1]{Lisandro Ferm\'in}
\author[1]{Tania Roa}
\author[1]{Soledad Torres}

\affil[1]{CIMFAV, Facultad de Ingenier\'ia, Universidad de Valpara\'iso}
\affil[2]{Instituto de Estad\'istica, Pontificia Universidad Cat\'olica de Valpara\'iso}

\affil[1]{\textit {hector.araya@postgrado.uv.cl}}
\affil[2]{\textit {natalia.bahamonde@pucv.cl}}
\affil[1]{\textit {lisandro.fermin@uv.cl}}
\affil[1]{\textit {tania.roa@postgrado.uv.cl}}
\affil[1]{\textit {soledad.torres@uv.cl}}

\usepackage{microtype}

\begin{document}
\maketitle
\abstract{In this article, we present the least squares estimator for the drift parameter in a linear regression model driven by the increment of a fractional Brownian motion sampled at random times. For two different random times, Jittered and renewal process sampling, consistency of the estimator is proven. A simulation study is provided to illustrate  the performance of the estimator under different values of the Hurst parameter  $H$.} \\
\textbf{Key words}: fractional Brownian motion, long memory, least squares estimator, random times, regression model.

\section{Introduction}
In different research areas, such as finance, network, meteorology, and astronomy among others, it has been noticed that the observations can be carried out sampling with random disturbances. Some  examples  of this  sampling are the  data behavior until it is necessary to increase the sampling frequency, measurements obtained at random times, and defining stopping time when a particular event occurs, etc.  For example, in \cite{nieto2015} the authors studied a Bayesian interpolation of unequally spaced time series. The case of paleoclimate time series was considered in \cite{max2014} and in \cite{olaf2016} where it is possible to estimate the significance of cross-correlations in unevenly sampled astronomical time series. Finally,  in the area of computer  science, we can mention the works given in \cite{chang2014} and \cite{zhao2014}. \\

The study of statistical models in those situations is quite promising and  has some open problems such as statistical inference  and the limit behavior of the estimators. In this article  we propose taking a first step in this direction; to study a simple regression model with Gaussian and long memory  noise, and  observation  measurements at random times. We consider here specific random times along with Jittered and Renewal Process sampling that we define properly in section \ref{Pre}. The term  {\it jitter} is related to the temporal variability during the sending of digital signals or as the small variation in the accuracy of the clock signal, see \cite{bell1981} and the references therein. It has also recently appeared in works related to the analysis of computational images, such as \cite{khan2017}, \cite{krune2016} and \cite{subr2014}. The case  of renewal process represents progressing randomness and distance from periodic sampling, see  \cite{durrett} for more details. 

Vilar et al. have written some previous works in this direction. In \cite{vilar}  the authors studied the nonparametric kernel estimator of the regression function, $m(x) = \mathbb{E}(Y ]X = x)$,  under mixing dependence conditions, and  the Ornstein-Uhlenbeck process driven by Brownian motion was studied in \cite{vilar2000}. 

Also, in \cite{masry1} the author studied the problem of estimating an unknown probability density function on the based on $n$ independent observations sampled at random times.

%

Using a wavelet analysis Bardet et al. in  \cite{bardet} studied the case of a nonparametric estimator of the spectral density of a Gaussian process with stationary increments, including the case of fractional Brownian motion,  from the observation of one path at  some particular class of random discrete times. They proved a central limit theorem and provided an application to biological data. Philippe et al. in \cite{PRV} gave one of the last works on this topic, where the authors considered the study of the preservation of memory  in a statistical model.

Our main purpose in studying a model with long memory noise  is the characterization of the strong correlations between observations or  persistence,  by a slow decay of the correlations. To explain this phenomenon in a model, it is common to represent it through the Hurst exponent $ H $, which takes values in $ [0,1] $. In particular, the long-range dependence can be seen when $H \in \left] 1/2, 1 \right] $. Since the work of Mandelbrot et al., \cite{mandelbrotfractional} the effect of long-range dependence has been studied over the years . One of the most popular stochastic processes with long memory  is the  fractional Brownian motion.   The demand of this process is caused by  a nice set of  properties, which are described below. A fractional Brownian motion  $B^H$ is a Gaussian process with  the following covariance structure
\begin{equation}\label{COVFRAC}
R_{H}(t, s) := \mathbb{E} \left[ B_{t}^{H} B_{s}^{H} \right] = \dfrac{\sigma^{2}}{2} \left[ |t|^{2H} + |s|^{2H} - |t-s|^{2H} \right].
\end{equation}

The family of processes $B^H_t : t \in [0, T]$ has  several properties such as:
\begin{enumerate}
\item The covariance of the increments of $B^H$ on intervals decays asymptotically as a negative
power of the distance between intervals.
\item Fractional Brownian motion is the only finite-variance process which is self-similar (with
index H) and has stationary increments.
\end{enumerate}

Those characteristics have converted the fractional Brownian motion into  one of the most natural generalizations of Brownian motion among the probability community.\\

With those motivations in mind, let us proceed to the mathematical description of the
model we are dealing with. Namely, we consider the following  simple regression  model 
\begin{eqnarray}
Y_{\tau_{i}} = a \tau_{i} + \Delta B_{\tau_{i}}^{H} , \; i =0, \dots, N(T),      
\label{modelo}
\end{eqnarray}
where $\Delta B_{\tau_{i+1}}^{H} = B_{\tau_{i+1}}^{H} - B_{\tau_{i}}^{H}$, 
 $T$ is a  positive fixed number, $N(T)$ means  the number of times that remain in the interval $[0, T]$. If $T=1$ we put $N(1):=N$. $\tau := \{ \tau_i , 1 \leq i \}$ are the random times given by Jittered sampling or associated to Renewal processes defined  in section  \ref{Pre}.
It is not hard to prove  that in the case of deterministic times, $\tau_{i} = \frac{Ti}{N}, 1 \leq i \leq N$ the $L^2$-consistency of the least squares estimation for the parameter $a$ is ensured. In fact, the rate of convergence is of order $N^2$.

The primary interest in this work is the parametric estimation and convergence results in the linear regression model, with long memory noise and  observations sampled at random times. It is important to recall that model the process $Y:= \{ Y_{\tau_i}, 1 \leq i \}$ defined by equation (\ref{modelo}) has long range dependence and is non-stationary in the weak sense.\\

Finally, this work is organized as follows: in Section 2 we present the definitions of the random times under which we will work and the model in which we will estimate the  parameter. In Section 3, we use the least squares procedure to obtain the parameter estimation, convergence results as $L^{p}$ and almost surely using jittered sampling and renewal process structure. To conclude, in Section 4, a simulation study is presented to illustrate the performance of both estimators, taking different values of H.

\section{Preliminaries}\label{Pre}
In this section, we introduce the main tools from the stochastic calculus  needed in the sequel. We present the fractional Brownian motion evaluated at two random times that we will consider throughout this work.
 Finally, the linear regression model and the least square estimator are presented.\\

Fractional Brownian motion $B^{H}$ with Hurst parameter $H \in (1/2, 1)$ is a centered Gaussian process with covariance structure is given in (\ref{COVFRAC}).
It is well-known that if $H = 1/2$, then $B^H$ is a standard Brownian motion. Also, the process $B^H$ is not a semimartingale if $H \neq 1/2$. Hence, we cannot apply the classical It\^o calculus to $B^{H}$.\\

Let  $T=1$ and $ \tau = \lbrace  \tau_{i}; i=0,\ldots,N \rbrace$ a strictly increasing sequence of random points over time, where  $N$ is the last integer such that $\tau_{N-1} \leq 1$, which exhibits one of the following two features.

\begin{enumerate}
\item { \bf Jittered sampling}. First, we assume that we observed a certain process at regular times  $\tau$ with period $\delta = 1/N >0$ but contaminated by an additive noise  $\nu$ which represents possible measurement errors. Then the sequence of random times $\tau_i, \quad 1 \le i \le N$ satisfies
\begin{eqnarray} \label{io}
\tau_{i} = \dfrac{i}{N} + \nu_{i}, \quad i = 1, \dots, N \; and\; \tau_0 := 0,
\end{eqnarray}
where $\lbrace \nu_{i}; \quad 1 \le i \le N \rbrace$   are independent and identically distributed set of random variables with density function $g(t)$, which is assumed to be symmetric. In the rest of the paper, it is assumed  that  $\nu_i \sim U \left[ -\frac{1}{2N}, \frac{1}{2N} \right]$ for all $i= 1,\ldots , N$, where $U[a,b]$ corresponds to the Uniform random variable on the interval $[a , b]$. 

\item { \bf Renewal process}. In this case, the sequence $\tau$ satisfies the renewal property, i.e.
\begin{eqnarray}
\tau_{i} = \sum_{j=1}^{i} t_j   \    \    \   \  i=1,2,... \    \    \  \mbox{and} \, \tau_{0} := 0,
\label{rp}
\end{eqnarray}
where $  \lbrace t_j  , 1 \le j \rbrace$  is a sequence of independent and identically distributed
random variables, with a common distribution function  $G(t)$ with support in $[0,\infty)$. Through the rest of this work it is assumed  that  $G(t)$ is an exponential distribution with parameter $N$ (number of observations). We will use the fact that an exponential distribution with parameter $N$ is equivalent to a gamma distribution $Gamma(1,N)$ and the sum of independent exponential random variables is a gamma random variable.
\end{enumerate} 


 Let us consider the random sampled linear Regression Model with Long Memory Noise defined in (\ref{modelo})
For the estimation of the parameter of interest in model the (\ref{modelo}), the least squares estimator  is computed and is determined by 

\begin{equation}\label{a-jt}
\hat{a}_{N} = \dfrac{\sum_{i=0}^{N-1} \tau_{i+1} Y_{\tau_{i+1}}}{\sum_{i=0}^{N-1} \tau_{i+1}^2}.
\end{equation}

%

\section{Main results}\label{main}
In this section, we provide our main results. First, we study the parameter estimation for  the random sampled linear Regression Model (\ref{modelo}) with random times  given by Jittered sampling. We prove  that $\hat{a}_{N}$ is an unbiased, $L^{p}$ - consistent  estimator for $a$. The same is proven in the case of renewal type observations. 
To study the asymptotic behavior of (\ref{DN}), we will separately analyse the numerator and the denominator.

\begin{rem}
It is worth mentioning that all the results  in this article can be extended to a noise  with the same covariance  structure as the fractional Brownian motion, such as Rosenblatt, Hermite  and fractional Poisson process. 

\end{rem}

\begin{thm} \label{as-io-rp}
Let  $\tau$ be given by  (\ref{io}) or (\ref{rp}). Then, the LS estimator $\hat{a}_N $ given in (\ref{a-jt}) of the drift parameter $a$ in the model (\ref{modelo}) is strongly consistent, that is
\begin{eqnarray*}
\hat{a}_N  \xrightarrow[N \to \infty]{a.s.}  a.
\end{eqnarray*}
\end{thm}

\begin{proof}
Recall that, from (\ref{modelo}) and (\ref{a-jt}) we have 
\begin{equation}\label{DN} 
\hat{a}_{N} - a  =  \dfrac{\dfrac{1}{N} \sum_{i=0}^{N-1} \tau_{i+1} \Delta B_{\tau_{i+1}}^H}{\dfrac{1}{N} \sum_{i=0}^{N-1} \tau_{i+1}^2} := \dfrac{A_N}{D_N}.
\end{equation} 
To prove our main theorem we need an auxiliary lemma relating with the almost surely convergence to the denominator $D_N$.  The proof of this lemma is given in Appendix \ref{AL}.

\begin{lem} \label{lem1}
Let $D_N$ be defined in (\ref{DN}). If $\tau_i $ are the sampling random times defined in  (\ref{io}) or (\ref{rp}), then 

\begin{equation}
D_{N} \xrightarrow[N \to \infty]{a.s.} \frac{1}{3}.
\end{equation}

\end{lem}

Hence, by Lemma \ref{lem1}, it remains to study the asymptotic behavior  of  $A_N$  as  $N \rightarrow \infty$.  It is quite easy to see by the definition of $A_N$ and conditioning on $\tau$, that $\mathbb{E} \left[ A_{N} \right]  = 0$. \\
Let us compute $\mathbb{E} \left[ A_{N} ^2 \right] $.    
\begin{eqnarray}\label{3c}
\mathbb{E} \left[ A_{N} ^2 \right] &= & \mathbb{E} \left[ \dfrac{1}{N^2}  \sum_{i=0}^{N-1} \tau_{i+1}^2 \left( B_{\tau_{i+1}}^{H} - B_{\tau_{i}}^{H}  \right)^{2} \right] \nonumber  \\
&+& \mathbb{E} \left[ \dfrac{1}{N^{2}} \sum_{i,j=1, \dots, N; |i-j|=1} \tau_{i+1} \tau_{j+1} \left( B_{\tau_{i+1}}^{H} - B_{\tau_{i}}^{H} \right) \left( B_{\tau_{j+1}}^{H} - B_{\tau_{j}}^{H} \right) \right] \nonumber  \\
&+& \mathbb{E} \left[ \dfrac{1}{N^{2}} \sum_{i,j=1, \dots, N; |i-j|\geq 2} \tau_{i+1} \tau_{j+1} \left( B_{\tau_{i+1}}^{H} - B_{\tau_{i}}^{H} \right) \left( B_{\tau_{j+1}}^{H} - B_{\tau_{j}}^{H} \right) \right] \nonumber  \\
&:=& \mathbb{E}( A_N^{(1)}) + \mathbb{E} (A_N^{(2)}) + \mathbb{E} (A_N^{(3)}), \label{sums} 
\end{eqnarray}
where we split  the sum into three terms associated to the  distance  of the indexes.

{\bf Jittered sampling case}\\
Let us study the  first term in (\ref{3c}) in the case of the Jittered sampling times defined in  (\ref{io}).
\begin{eqnarray}
 \mathbb{E}( A_N^{(1)}) &=& \dfrac{1}{N^2} \left. \sum_{i=0}^{N-1} \mathbb{E} \left[ \mathbb{E} \left[ \left( \dfrac{i+1}{N} + \nu_{i+1} \right)^2 \left( B_{ \frac{i+1}{N} + \nu_{i+1} }^{H} - B_{\frac{i}{N} + \nu_{i}}^{H} \right)^{2} \right| \nu_{i} = s_{i} \, , \nu_{i+1} = s_{i+1} \right] \right] \nonumber \\
&=& \dfrac{1}{N^2} \sum_{i=0}^{N-1} \int_{-\frac{1}{2N}}^{\frac{1}{2N}} \int_{-\frac{1}{2N}}^{\frac{1}{2N}} \left( \dfrac{i+1}{N} + s_{i+1} \right)^2 \mathbb{E} \left[ \left( B_{\frac{i+1}{N} + s_{i+1} }^{H} - B_{ \frac{i}{N} + s_{i} }^{H} \right)^2 \right]  f_{\nu_{i}, \nu_{i+1}} (s_{i}, s_{i+1}) d s_{i} d s_{i+1} \nonumber \\
&=& \dfrac{1}{N^2} \sum_{i=0}^{N-1} \int_{-\frac{1}{2N}}^{\frac{1}{2N}} \int_{-\frac{1}{2N}}^{\frac{1}{2N}} \left( \dfrac{i+1}{N} + s_{i+1} \right)^2 \left( \dfrac{i+1}{N} + s_{i+1} - \dfrac{i}{N} - s_{i} \right)^{2H} N^2 d s_{i} d s_{i+1} \nonumber \\
&=& \dfrac{1}{N^2} \sum_{i=0}^{N-1} \int_{-\frac{1}{2N}}^{\frac{1}{2N}} \int_{-\frac{1}{2N}}^{\frac{1}{2N}} \left( \dfrac{i+1}{N} + s_{i+1} \right)^2 \left( \dfrac{1}{N} + s_{i+1} - s_{i} \right)^{2H} N^2 d s_{i} d s_{i+1} , \label{AJ}
\end{eqnarray}
where $f_{\nu_{i}, \nu_{i+1}} (s_{i}, s_{i+1})$ is the joint distribution of the  couple $(\nu_i, \nu_{i+1})$ which is the product of two independent Uniform$[-1/2N , 1/2N] $ random variables.    Since $s_{i} \in [-1/2N, 1/2N ]$ we obtain
\begin{eqnarray}\label{A1-1}
\mathbb{E}( A_N^{(1)}) &\leq& \dfrac{1}{N^2} \sum_{i=0}^{N-1} \int_{-\frac{1}{2N}}^{\frac{1}{2N}} \int_{-\frac{1}{2N}}^{\frac{1}{2N}} \left( \dfrac{i+1}{N} + \dfrac{1}{2N} \right)^2 \left( \dfrac{1}{N} + \dfrac{1}{N} \right)^{2H} N^2 d s_{i} d s_{i+1} \nonumber \\
& \leq & \dfrac{2^{2H-2}}{N^{4+2H}}  (N+1)(N+2)(2N+3) \le \frac{C_1(H)}{N^{1+2H}} .
\end{eqnarray}
 Let us consider the case  of $|i-j|=1$ in (\ref{sums}). For simplicity we will take  $j=i-1$ the other case can be treated in a similar way. Therefore 
\begin{align}\label{A2}
\mathbb{E}(A_N^{(2)} )&= \frac{2}{N^{2}} \sum_{i=1}^{N-1} \mathbb{E} \left[ \tau_{i+1} \tau_{i} \left( B_{\tau_{i+1}}^{H} - B_{\tau_{i}}^{H} \right) \left( B_{\tau_{i}}^{H} - B_{\tau_{i-1}}^{H} \right) \right] \nonumber \\
&= \frac{2}{N^{2}} \sum_{i=1}^{N-1} \int_{-\frac{1}{2N}}^{\frac{1}{2N}} \int_{-\frac{1}{2N}}^{\frac{1}{2N}} \int_{-\frac{1}{2N}}^{\frac{1}{2N}} \left( \frac{i+1}{N} + s_{i+1} \right) \left( \frac{i}{N} + s_{i} \right)  \nonumber\\
& \times \mathbb{E} \left[ \left( B_{\frac{i+1}{N} + s_{i+1}}^{H} - B_{\frac{i}{N} + s_{i}}^{H} \right) \left( B_{\frac{i}{N} + s_{i}}^{H} - B_{\frac{i-1}{N} + s_{i-1}}^{H} \right) \right] N^{3} ds_{i-1} ds_{i} ds_{i+1},
\end{align}
where we conditioned with respect to $\nu_{i-1}=s_{i-1}$, $\nu_{i}=s_{i}$ and $\nu_{i+1}=s_{i+1}$. 
Since, 
\begin{align*}
& \mathbb{E} \left[ \left( B_{\frac{i+1}{N} + s_{i+1}}^{H} - B_{\frac{i}{N} + s_{i}}^{H} \right)  \left( B_{\frac{i}{N} + s_{i}}^{H} - B_{\frac{i-1}{N} + s_{i-1}}^{H} \right) \right] \\
&= \frac{1}{2} \left[ \left| s_{i+1} - s_{i-1} + \frac{2}{N} \right|^{2H} - \left| s_{i+1} - s_{i} + \frac{1}{N} \right| ^{2H} - \left| s_{i} - s_{i-1} + \frac{1}{N} \right|^{2H} \right].
\end{align*}
 Then
\begin{align}\label{ICov}
\mathbb{E} \left[ \left( B_{\frac{i+1}{N} + s_{i+1}}^{H} - B_{\frac{i}{N} + s_{i}}^{H} \right) \left( B_{\frac{i}{N} + s_{i}}^{H} - B_{\frac{i-1}{N} + s_{i-1}}^{H} \right) \right] & \leq  \frac{1}{2}   \left({\frac{3}{N}}\right)^{2H}  \le \frac{C_2(H)} {2 N^{2H}},  \text{with}  \ C_2(H)= 3^{2H}.
\end{align}
Now, plugging inequality (\ref{ICov}) into the equation (\ref{A2}) yields
\begin{align}\label{A2-1}
\mathbb{E}(A_N^{(2)}) & \leq \frac{2}{N^{2}} \sum_{i=1}^{N-1} \int_{-\frac{1}{2N}}^{\frac{1}{2N}} \int_{-\frac{1}{2N}}^{\frac{1}{2N}} \int_{-\frac{1}{2N}}^{\frac{1}{2N}} \left( \frac{i+1}{N} + s_{i+1} \right) \left( \frac{i}{N} + s_{i} \right) \frac{C_2(H)}{2 N^{2H}} N^{3} ds_{i-1} ds_{i} ds_{i+1} \nonumber\\
& \leq \frac{C_2(H)}{ N^{2+2H}} \sum_{i=1}^{N-1}  \int_{-\frac{1}{2N}}^{\frac{1}{2N}} \int_{-\frac{1}{2N}}^{\frac{1}{2N}} \int_{-\frac{1}{2N}}^{\frac{1}{2N}} \left( \frac{i+1}{N} + \frac{1}{2N} \right) \left( \frac{i}{N} + \frac{1}{2N} \right) N^{3} ds_{i-1} ds_{i} ds_{i+1} \nonumber\\
&= \frac{C_2(H)}{ N^{2+2H}} \sum_{i=1}^{N-1} \left( \frac{i+1}{N} + \frac{1}{2N} \right) \left( \frac{i}{N} + \frac{1}{2N} \right) = \frac{C_2(H)}{ N^{2+2H}} \sum_{i=0}^{N-2} \left( \frac{2i + 3}{2N} \right)^{2} \nonumber\\
&\le \frac{C_2(H)}{ 6 N^{2H+4}}  \left(N(N+1)(2N  +1) \right) \le  \frac{C_3(H)}{ N^{2H+1}} .   
\end{align}
Let us consider the case  $|i-j|=2$ in (\ref{sums}). Conditioning on  $\nu_{i}=s_{i} , \nu_{i+1}=s_{i+1}, \nu_{j}=s_{j}$  and  $\nu_{j+1}=s_{j+1}$, we get   
\begin{align}\label{A3}
\mathbb{E}(A_N^{(3)}) &= \frac{1}{N^{2}} \mathbb{E} \left[ \sum_{i,j=1, \dots, N; |i-j| \geq 2} \tau_{i+1} \tau_{j+1} \left( B_{\tau_{i+1}}^{H} - B_{\tau_{i}}^{H} \right) \left( B_{\tau_{j+1}}^{H} - B_{\tau_{j}}^{H} \right) \right] \nonumber \\
&= \frac{1}{N^{2}} \sum_{i,j=1, \dots, N; |i-j| \geq 2} \int_{-\frac{1}{2N}}^{\frac{1}{2N}} \int_{-\frac{1}{2N}}^{\frac{1}{2N}} \int_{-\frac{1}{2N}}^{\frac{1}{2N}} \int_{-\frac{1}{2N}}^{\frac{1}{2N}} \left( \frac{i+1}{N} + s_{i+1} \right) \left( \frac{j+1}{N} + s_{j+1} \right)  \nonumber \\
& \times \mathbb{E} \left[ \left( B_{ \frac{i+1}{N} + s_{i+1} }^{H} - B_{\frac{i}{N} + s_{i} }^{H} \right)  \left(B_{ \frac{j+1}{N} + s_{j+1} }^{H} -  B_{ \frac{j}{N} + s_{j} }^{H} \right) \right] 
N^{4} ds_{i} ds_{i+1} ds_{j} ds_{j+1}.
\end{align}
Let 
\begin{align*}
& {\bf I}:=\mathbb{E} \left[ \left( B_{ \frac{i+1}{N} + s_{i+1}}^{H} - B_{ \frac{i}{N} + s_{i})}^{H} \right) \left( B_{ \frac{j+1}{N} + s_{j+1} }^{H} - B_{ \frac{j}{N} + s_{j} }^{H} \right) \right] \\
&= \frac{1}{2} \left[ \left| \frac{i-j+1}{N} + s_{i+1} - s_{j} \right| ^{2H}  + \left| \frac{i-j-1}{N} + s_{i} - s_{j+1} \right| ^{2H} \right. \\
& -\left. \left| \frac{i-j}{N} + s_{i+1} - s_{j+1} \right| ^{2H} - \left| \frac{i-j}{N} + s_{i} - s_{j} \right| ^{2H} \right].
\end{align*}
By Taylor theorem applied to  the function $x^{2H}$ allows to get  
\begin{equation*}
\left\vert \frac{i-j+1}{N} + s_{i+1} - s_{j} \right\vert ^{2H} -  \left\vert \frac{i-j}{N} + s_{i} - s_{j} \right\vert^{2H} = 2H  \left\vert \frac{i-j}{N} + s_{i} - s_{j} \right\vert^{2H-1} \left( s_{i+1}-s_{i} + \dfrac{1}{N} \right) + R^{1}_{N}
\end{equation*} 
and 
\begin{equation*}
\left\vert \frac{i-j}{N} + s_{i+1} - s_{j+1} \right\vert ^{2H} -  \left\vert \frac{i-j-1}{N} + s_{i} - s_{j+1} \right\vert^{2H} = 2H  \left\vert \frac{i-j-1}{N} + s_{i} - s_{j+1} \right\vert^{2H-1} \left( s_{i+1}-s_{i} + \dfrac{1}{N} \right) + R^{2}_{N}.
\end{equation*} 
Therefore 
\begin{equation*}
{\bf I}= H \left( s_{i+1}-s_{i} + \dfrac{1}{N} \right) \left[  \left\vert \frac{i-j}{N} + s_{i} - s_{j} \right\vert^{2H-1} - \left\vert \frac{i-j-1}{N} + s_{i} - s_{j+1} \right\vert^{2H-1}  \right] +  R^{1}_{N} - R^{2}_{N}.
\end{equation*}
Again, applying Taylor theorem to the function $x^{2H -1}$, we obtain 
\begin{equation*}
{\bf I} = H (2H-1) \left( s_{i+1}-s_{i} + \dfrac{1}{N} \right) \left( s_{j+1}-s_{j} + \dfrac{1}{N} \right)   \left\vert \frac{i-j-1}{N} + s_{i} - s_{j+1} \right\vert^{2H-2} + R^{3}_{N} +  R^{1}_{N} - R^{2}_{N},
\end{equation*}
which implies 
\begin{equation}\label{I-a}
{\bf I} \leq  \dfrac{C_4(H)}{N^2}  \left\vert \frac{i-j-1}{N} + s_{i} - s_{j+1} \right\vert^{2H-2}. 
\end{equation}

Notice here that the remainder terms  $R^{1}_{N} , R^{2}_{N}$ and  $R^{3}_{N}$ is of order $N^{-2H }$.
Plugging (\ref{I-a})  into  the expression (\ref{A3}) we obtain
\begin{eqnarray}\label{A3-1}
\mathbb{E}(A_{N}^{(3)})  &\leq & \frac{C_4(H)}{N^{4}} \sum_{i,j=1, \dots, N; |i-j| \geq 2} \left( \frac{i+1}{N} + s_{i+1} \right) \left( \frac{j+1}{N} + s_{j+1} \right) \left\vert \dfrac{i-j-1}{N} + s_{i} -s_{j+1}  \right\vert^{2H-2} \nonumber \\
&\leq & \frac{C_4(H)}{N^{2}}  \frac{1}{N^2} \sum_{i,j=1, \dots, N; |i-j| \geq 2} \left( \frac{i+2}{N}  \right) \left( \frac{j+2}{N}  \right) \left( \dfrac{i-j}{N}  \right)^{2H-2}.
\end{eqnarray}
Moreover, we can see the expression in (\ref{A3-1}) 
\begin{equation*}
\frac{1}{N^{2}} \sum_{i,j=1, \dots, N; |i-j| \geq 2} \left( \frac{i+2}{N}  \right) \left( \frac{j+2}{N}  \right)\left( \dfrac{i-j}{N}  \right)^{2H-2},
\end{equation*}
as a  Riemann sum of the double integral  $\int_0^1 \int_0^1 x y (x-y)^{2H -2} dxdy$ which is finite.  Then 
\begin{equation}\label{A3-2}
\mathbb{E}(A_{N}^{(3)})  \leq  \frac{C_5(H)}{N^{2}}.
\end{equation}

Finally,  substituting (\ref{A1-1})    (\ref{A2-1})  (\ref{A3-2})  into  the equation in  (\ref{sums})  we obtain 

\begin{eqnarray*}
\mathbb{E} \left[  A_{N} ^2 \right] &\le & \frac{C_1(H) + C_3(H)}{N^{1+2H}}  +   \frac{C_5(H)}{N^{2}}.
\end{eqnarray*}

Since the $L^2$ rate of $A_{N}$ is faster than $1/N$. A direct application of Borell-Cantelli lemma allow us to obtain
\begin{equation}
A_N  \xrightarrow[N \to \infty]{a.s.}  0.
\end{equation}

\vspace{1cm}

\noindent {\bf Renewal sampling case}\\
Let us consider now the case where $\tau_{i} $ are the sampling times defined from (\ref{rp}). In this case is easy to see that  
$\mathbb{E} \left[ A_{N} \right] =0$.
Let us compute now the second moment of $A_{N}$.
\begin{eqnarray}\label{Desc}
\mathbb{E} \left[ A_{N}^2 \right] & = & \dfrac{1}{N^2} \mathbb{E} \left[ \left( \sum_{i=0}^{N-1} \tau_{i+1} \left( B_{\tau_{i+1}}^{H} - B_{\tau_{i}}^{H} \right) \right)^{2} \right]  \nonumber \\
&=&\dfrac{1}{N^2} \mathbb{E} \left[  \sum_{i=0}^{N-1} \tau_{i+1}^{2} \left( B_{\tau_{i+1}}^{H} - B_{\tau_{i}}^{H} \right)^{2} \right] \nonumber \\
& + & \dfrac{1}{N^2} \mathbb{E} \left[ \sum_{i,j=1, \dots, N ; |i-j|=1} \tau_{i+1} \tau_{j+1} \left( B_{\tau_{i+1}}^{H} - B_{\tau_{i}}^{H} \right) \left( B_{\tau_{j+1}}^{H} - B_{\tau_{j}}^{H} \right) \right]  \nonumber \\
& + & \dfrac{1}{N^2} \mathbb{E} \left[ \sum_{i,j=1, \dots, N ; |i-j| \geq 2} \tau_{i+1} \tau_{j+1} \left( B_{\tau_{i+1}}^{H} - B_{\tau_{i}}^{H} \right) \left( B_{\tau_{j+1}}^{H} - B_{\tau_{j}}^{H} \right) \right]  \nonumber  \\
&=& A_{N}^{(1)} + A_{N}^{(2)} + A_{N}^{(3)},
\end{eqnarray}
where we split in three terms  as in the case of  jittered random times. \\
First, let us analyse $A_N^{(1)}$.
\begin{eqnarray*}
A_{N}^{(1)} &=& \left. \dfrac{1}{N^2} \sum_{i=0}^{N-1} \mathbb{E} \left[ \mathbb{E} \left[ \tau_{i+1}^{2} \left( B_{\tau_{i+1}}^{H} - B_{\tau_{i}}^{H} \right)^{2} \right| \tau_{i} = z_{i} , \tau_{i+1} = z_{i+1} \right] \right] \\
&=& \dfrac{1}{N^2} \sum_{i=0}^{N-1} \int_{0}^{\infty} \int_{0}^{z_{i+1}} z_{i+1}^{2} \mathbb{E} \left[ \left( B_{z_{i+1}}^{H} - B_{z_{i}}^{H} \right)^{2} \right] f_{\tau_{i}, \tau_{i+1}} (z_{i}, z_{i+1}) d z_{i} d z_{i+1} .
\end{eqnarray*}
According to  the  probability density function for the 2-dimensional random vector $(\tau_ i,\tau_{i+1})$ given in  Appendix  \ref{A1}, we have 
\begin{align}\label{I}
A_N^{(1)} &= \dfrac{1}{N^2} \sum_{i=0}^{N-1} \int_{0}^{\infty} \int_{0}^{z_{i+1}} z_{i+1}^{2} \left( z_{i+1} - z_{i} \right)^{2H} \dfrac{N^{i+1}}{\Gamma(i)} z_{i}^{i-1} e^{-N z_{i+1}} d z_{i} d z_{i+1} \nonumber \\
&\le \frac{ C_6(H)}{N^{2H+1}}. 
\end{align}
To estimates $A_{N}^{(2)}$, we take into account that $|i-j|=1$. Then  
\begin{align}\label{tau2}
A_{N}^{(2)} &= \dfrac{2}{N^{2}} \mathbb{E} \left[ \sum_{i=1}^{N-1} \tau_{i+1} \tau_{i} \left( B_{\tau_{i+1}}^{H} - B_{\tau_{i}}^{H} \right) \left( B_{\tau_{i}}^{H} - B_{\tau_{i-1}}^{H} \right) \right]. 
\end{align}
Note that for  $i=1$ and since $\tau_0=0$, the first term in the sum (\ref{tau2}) denoted by $A_{N,1}^{(2)}$ is equal to  $\frac{1}{N^2} \mathbb{E}  \left[\tau_{2} \tau_{1} \left( B_{\tau_{2}}^{H} - B_{\tau_{1}}^{H} \right) B_{\tau_{1}}^{H}\right]$. This term can be solved in the same way that it was made in $A_N^{(1)}$.

We define  $A_{N,(-1)}^{(2)}$ as the sum in (\ref{tau2}) without the term $A_{N,1}^{(2)}$. Conditioning on $\tau_{i-1},\tau_{i} ,\tau_{i+1}$ and considering the  probability distribution for the 3-valued random vector given in Appendix \ref{A1} we get 
%
\begin{align*}
A_{N,(-1)}^{(2)} &= \dfrac{2}{N^{2}} \sum_{i=2}^{N-1}  \mathbb{E} \left[ \tau_{i+1} \tau_{i} \left( B_{\tau_{i+1}}^{H} - B_{\tau_{i}}^{H} \right) \left( B_{\tau_{i}}^{H} - B_{\tau_{i-1}}^{H} \right) \right] \\
&= \dfrac{2}{N^{2}} \sum_{i=2}^{N-1} \int_{0}^{\infty} \int_{0}^{z_{i+1}} \int_{0}^{z_{i}} z_{i} z_{i+1} \mathbb{E} \left[ \left( B_{z_{i+1}}^{H} - B_{z_{i}}^{H} \right) \left( B_{z_{i}}^{H} - B_{z_{i-1}}^{H} \right) \right] f_{z_{i-1}, z_{i}, z_{i+1}} d_{z_{i-1}} d_{z_{i}} d_{z_{i+1}}.
\end{align*}
Noting that  
\begin{align}
& \mathbb{E} \left[ \left( B_{z_{i+1}}^{H} - B_{z_{i}}^{H} \right) \left( B_{z_{i}}^{H} - B_{z_{i-1}}^{H} \right) \right]
= \dfrac{1}{2} \left[ \left( z_{i+1} - z_{i-1} \right) ^{2H} - \left( z_{i+1} - z_{i} \right) ^{2H} - \left( z_{i} - z_{i-1} \right) ^{2H} \right],
\end{align}
we obtain 
\begin{align*}
\left \vert A_{N,(-1)}^{(2)} \right \vert & \leq \dfrac{1}{N^{2}} \sum_{i=2}^{N-1} \int_{0}^{\infty} \int_{0}^{z_{i+1}} \int_{0}^{z_{i}} z_{i} z_{i+1} \left( z_{i+1} - z_{i-1} \right) ^{2H} \dfrac{N^{i+1}}{\Gamma(i-1)} z_{i-1}^{i-2} e^{-N z_{i+1}} dz_{i-1} dz_{i} dz_{i+1} \\
&+ \dfrac{1}{ N^{2}} \sum_{i=2}^{N-1} \int_{0}^{\infty} \int_{0}^{z_{i+1}} \int_{0}^{z_{i}} z_{i} z_{i+1} \left( z_{i+1} - z_{i} \right) ^{2H} \dfrac{N^{i+1}}{\Gamma(i-1)} z_{i-1}^{i-2} e^{-N z_{i+1}} dz_{i-1} dz_{i} dz_{i+1} \\
&+ \dfrac{1}{ N^{2}} \sum_{i=2}^{N-1} \int_{0}^{\infty} \int_{0}^{z_{i+1}} \int_{0}^{z_{i}} z_{i} z_{i+1} \left( z_{i} - z_{i-1} \right) ^{2H} \dfrac{N^{i+1}}{\Gamma(i-1)} z_{i-1}^{i-2} e^{-N z_{i+1}} dz_{i-1} dz_{i} dz_{i+1} \\
&= A_{N,(-1)}^{(2,1)} + A_{N,(-1)}^{(2,2)} + A_{N,(-1)}^{(2,3)}.
\end{align*}
We estimate separetely  the terms  $A_{N,(-1)}^{(2,1)}$,  $A_{N,(-1)}^{(2,2)}$ and $ A_{N,(-1)}^{(2,3)}$.
For  $A_{N,(-1)}^{(2,2)}$ we have the following estimates 
\begin{align}\label{tauA22}
A_{N,(-1)}^{(2,2)}  &= \dfrac{1}{ N^{2}} \sum_{i=2}^{N-1} \int_{0}^{\infty} \int_{0}^{z_{i+1}} \int_{0}^{z_{i}} z_{i} z_{i+1} \left( z_{i+1} - z_{i} \right) ^{2H} \dfrac{N^{i+1}}{\Gamma(i-1)} z_{i-1}^{i-2} e^{-N z_{i+1}} dz_{i-1} dz_{i} dz_{i+1} \nonumber \\
&= \dfrac{1}{ N^{2}} \sum_{i=2}^{N-1} \int_{0}^{\infty} \int_{0}^{z_{i+1}}z_{i}^{i} z_{i+1} \left( z_{i+1} - z_{i} \right) ^{2H} \dfrac{N^{i+1}}{\Gamma(i)} e^{-N z_{i+1}}  dz_{i} dz_{i+1} \nonumber \\
&= \dfrac{1}{ N^{2}} \sum_{i=2}^{N-1} \int_{0}^{\infty} z_{i+1}\dfrac{N^{i+1}}{\Gamma(i)} e^{-N z_{i+1}}  \int_{0}^{z_{i+1}}z_{i}^{i}  \left( z_{i+1} - z_{i} \right) ^{2H}  dz_{i} dz_{i+1}.
\end{align}
By (\ref{inte1}) in Appendix \ref{A1} the last integral in (\ref{tauA22}) is 
$$
\int_{0}^{z_{i+1}}  z_{i}^{i}  \left( z_{i+1} - z_{i} \right) ^{2H}  dz_{i} =  \frac{\Gamma(2H+1) \Gamma(i+1)z_{i+1}^{2H + i + 1} }{\Gamma(2H + i +2)   }.
$$
Plugging the last equality into equation (\ref{tauA22}) we obtain 
\begin{align}\label{tauA22-1}
A_{N,(-1)}^{(2,2)}  &= \dfrac{1}{ N^{2}} \sum_{i=2}^{N-1}\frac{\Gamma(2H+1) \Gamma(i+1) }{\Gamma(2H + i +2)   }   \dfrac{N^{i+1}}{\Gamma(i)}\int_{0}^{\infty} z_{i+1}^{2H +i+2} e^{-N z_{i+1}}  
dz_{i+1} \nonumber \\
  &= \dfrac{\Gamma(2H+1)}{ N^{2}} \sum_{i=2}^{N-1}\frac{i N^{i+1} }{\Gamma(2H + i +2)   }  \int_{0}^{\infty} z_{i+1}^{2H +1+2} e^{-N z_{i+1}}  dz_{i+1}.
\end{align}
By (\ref{inte2}) in Appendix \ref{A1} the last integral in (\ref{tauA22-1}) is 
$$
\int_{0}^{\infty} z_{i+1}^{2H +i+2} e^{-N z_{i+1}}  dz_{i+1}= N^{-2H -i-3} \Gamma(2H +i+3).
$$
Plugging the last equality into the equation (\ref{tauA22-1}) gives
\begin{align}\label{tauA22-2}
A_{N,(-1)}^{(2,2)}  &= \dfrac{\Gamma(2H+1)}{ N^{2H+4}} \sum_{i=2}^{N-1}  i (2H+i+2) \le \frac{C_7(H)}{N^{2H+1}}. 
\end{align}
Let set $x = z_{i+1} - z_{i}$ and  $y = z_{i} - z_{i-1}$. It is clear that both $x,y \ge 0$. 
Since $1 \leq 2H \leq 2$, then $f(x) = x^{2H}$ is a convex function. This implies  $f(x+y) \leq \frac{1}{2} \left[ f(2x) + f(2y) \right]$.  Then 
\begin{align*}
A_{N,(-1)}^{(2,1)} &\le \dfrac{2^{2H-1} }{ N^{2}} \sum_{i=2}^{N-1} \dfrac{N^{i+1}}{\Gamma(i-1)} \int_{0}^{\infty} z_{i+1} e^{-N z_{i+1}} \int_{0}^{z_{i+1}} z_{i} \int_{0}^{z_{i}} \left( z_{i+1} - z_{i} \right)^{2H} z_{i-1}^{i-2} dz_{i-1} dz_{i} dz_{i+1}  \\
&+ \dfrac{2^{2H-1} }{ N^{2}} \sum_{i=2}^{N-1} \dfrac{N^{i+1}}{\Gamma(i-1)} \int_{0}^{\infty} z_{i+1} e^{-N z_{i+1}} \int_{0}^{z_{i+1}} z_{i} \int_{0}^{z_{i}} \left( z_{i} - z_{i-1} \right)^{2H} z_{i-1}^{i-2} dz_{i-1} dz_{i} dz_{i+1}  \\ 
&= A_{N,(-1)}^{(2,1,1)} + A_{N,(-1)}^{(2,1,2)}.
\end{align*}
We analyse first the term  $A_{N,(-1)}^{(2,1,1)}$.
\begin{align}\label{tauA211}
A_{N,(-1)}^{(2,1,1)} 
&= \dfrac{2^{2H-1}}{2 N^{2}} \sum_{i=2}^{N-1} \dfrac{N^{i+1}}{\Gamma(i)} \int_{0}^{\infty} z_{i+1} e^{-N z_{i+1}} \int_{0}^{z_{i+1}} z_{i}^{i} \left( z_{i+1} - z_{i} \right)^{2H} dz_{i} dz_{i+1}.
\end{align}
By (\ref{inte1}) in Appendix \ref{A1}, the equality (\ref{tauA211}) becomes
\begin{align}\label{tauA211-1}
A_{N,(-1)}^{(2,1,1)} &= \dfrac{2^{2H-1}}{ N^{2}} \sum_{i=2}^{N-1} \dfrac{N^{i+1}}{\Gamma(i)} \int_{0}^{\infty} z_{i+1} e^{-N z_{i+1}} \dfrac{\Gamma(2H+1) \Gamma(i+1)}{\Gamma(2H+i+2)} z_{i+1}^{2H+i+1} dz_{i+1} \nonumber \\
&= \dfrac{2^{2H-1} \Gamma(2H+1)}{ N^{2}} \sum_{i=2}^{N-1} i \dfrac{N^{i+1}}{\Gamma(2H+i+2)} \int_{0}^{\infty} z_{i+1}^{2H+i+2} e^{-N z_{i+1}} dz_{i+1}.
\end{align}
Invoking  (\ref{inte2}) in Appendix \ref{A1} we get the following estimates for (\ref{tauA211-1}),
\begin{align}\label{tauA211-}
A_{N,(-1)}^{(2,1,1)} &= \dfrac{2^{2H-1} \Gamma(2H+1)}{ N^{2}} \sum_{i=2}^{N-1} \frac{ i (2H+i+2)} {N^{2H+2}} \leq \frac{C_8(H)}{N^{2H+1}}.
\end{align}
Let us consider the term  $A_{N,(-1)}^{(2,1,2)}$. With the same techniques  as in the case $A_{N,(-1)}^{(2,1,1)}$ (see equation (\ref{tauA211})), and  taking into account (\ref{inte1}) and (\ref{inte2}) we obtain 
\begin{align}\label{tauA212}
A_{N,(-1)}^{(2,1,2)} &\le \frac{C_9(H)}{N^{2H+1}}.
\end{align}
Finally, we examine the term $A_N^{(3)}$ in (\ref{Desc}) 
\begin{align*}
A_N^{(3)} 
&= \dfrac{2}{N^{2}} \sum_{i,j=1, \dots, N; (i-j) \geq 2} \int_{0}^{\infty} \int_{0}^{z_{i+1}} \int_{0}^{z_{i}} \int_{0}^{z_{j+1}} z_{i+1} z_{j+1} \mathbb{E} \left[ \left(B_{z_{i+1}}^{H} - B_{z_{i}}^{H} \right) \left(B_{z_{j+1}}^{H} - B_{z_{j}}^{H} \right) \right] \\
&\times  f_{\tau_{j}, \tau_{j+1}, \tau_{i}, \tau_{i+1}}(z_{j}, z_{j+1}, z_{i}, z_{i+1}) dz_{j} dz_{j+1} dz_{i} dz_{i+1}.
\end{align*}
Note that 

\begin{align*}
\left| \mathbb{E} \left[ \left(B_{z_{i+1}}^{H} - B_{z_{i}}^{H} \right) \left(B_{z_{j+1}}^{H} - B_{z_{j}}^{H} \right) \right]  \right| & \leq \dfrac{1}{2} \left[(z_{i+1} - z_{j})^{2H} + (z_{i} - z_{j+1})^{2H} + (z_{i+1} - z_{j+1})^{2H} + (z_{i} - z_{j})^{2H}  \right] \\
& \leq \dfrac{1}{2} \left( 2^{2H-1}+1\right)^2 \left[ \left( z_{i+1} - z_{i} \right)^{2H} + \left( z_{i} - z_{j+1} \right)^{2H} +\left( z_{j+1} - z_{j} \right)^{2H} \right].
\end{align*}

%

According to the probability joint density function $f_{\tau_{j}, \tau_{j+1}, \tau_{i}, \tau_{i+1}}(z_{j}, z_{j+1}, z_{i}, z_{i+1})$ given in Appendix \ref{A1} we have the following estimates for  $\vert A_N^{(3)} \vert$: 

\begin{align*}
\vert A_N^{(3)} \vert &=  \dfrac{ \left( 2^{2H-1}+1\right)^2}{N^{2}} \sum_{i,j=1, \dots, N; (i-j) \geq 2} \int_{0}^{\infty} \int_{0}^{z_{i+1}} \int_{0}^{z_{i}} \int_{0}^{z_{j+1}} z_{i+1} z_{j+1} (z_{i+1} - z_{i})^{2H}   \\
& \dfrac{N^{i+1}}{\Gamma(j) \Gamma(i-j-1)} z_{j}^{j-1} (z_{i} - z_{j+1})^{i-j-2} e^{-N z_{i+1}} dz_{j} dz_{j+1} dz_{i} dz_{i+1} \\
&+ \dfrac{ \left( 2^{2H-1}+1\right)^2}{N^{2}} \sum_{i,j=1, \dots, N; (i-j) \geq 2} \int_{0}^{\infty} \int_{0}^{z_{i+1}} \int_{0}^{z_{i}} \int_{0}^{z_{j+1}} z_{i+1} z_{j+1} (z_{i} - z_{j+1})^{2H}  \\
& \dfrac{N^{i+1}}{\Gamma(j) \Gamma(i-j-1)} z_{j}^{j-1} (z_{i} - z_{j+1})^{i-j-2} e^{-N z_{i+1}} dz_{j} dz_{j+1} dz_{i} dz_{i+1} \\
&+ \dfrac{ \left( 2^{2H-1}+1\right)^2}{N^{2}} \sum_{i,j=1, \dots, N; (i-j) \geq 2} \int_{0}^{\infty} \int_{0}^{z_{i+1}} \int_{0}^{z_{i}} \int_{0}^{z_{j+1}} z_{i+1} z_{j+1} (z_{j+1} - z_{j})^{2H} \\
& \dfrac{N^{i+1}}{\Gamma(j) \Gamma(i-j-1)} z_{j}^{j-1} (z_{i} - z_{j+1})^{i-j-2} e^{-N z_{i+1}} dz_{j} dz_{j+1} dz_{i} dz_{i+1} \\
&= A_{N}^{(3,1)} + A_{N}^{(3,2)} + A_{N}^{(3,3)}.
\end{align*}
We first analyse  $A_{N}^{(3,1)}$. By using the expressions  given in Appendix  \ref{A1} by the formulas  (\ref{inte1}) and (\ref{inte2}) we have    
\begin{align}\label{A3-1-1}
A_{N}^{(3,1)} &= \dfrac{ \left( 2^{2H-1}+1\right)^2}{N^{2}} \sum_{i,j=1, \dots, N; (i-j) \geq 2} \dfrac{N^{i+1}}{\Gamma(j) \Gamma(i-j-1)} \int_{0}^{\infty} z_{i+1} e^{-N z_{i+1}} \int_{0}^{z_{i+1}} (z_{i+1} - z_{i})^{2H} \nonumber \\
& \int_{0}^{z_{i}} z_{j+1} (z_{i} - z_{j+1})^{i-j-2} \int_{0}^{z_{j+1}} z_{j}^{j-1}  dz_{j} dz_{j+1} dz_{i} dz_{i+1} \nonumber \\
&=  \dfrac{ \left( 2^{2H-1}+1\right)^2}{N^{2}} \sum_{i,j=1, \dots, N; (i-j) \geq 2} \dfrac{N^{i+1}}{\Gamma(j+1) \Gamma(i-j-1)} \int_{0}^{\infty} z_{i+1} e^{-N z_{i+1}} \int_{0}^{z_{i+1}} (z_{i+1} - z_{i})^{2H} \nonumber  \\
& \int_{0}^{z_{i}} z_{j+1}^{j+1} (z_{i} - z_{j+1})^{i-j-2} dz_{j+1} dz_{i} dz_{i+1} \nonumber \\
&= \dfrac{ \left( 2^{2H-1}+1\right)^2}{N^{2}} \sum_{i,j=1, \dots, N; (i-j) \geq 2} \dfrac{N^{i+1}}{\Gamma(j+1) \Gamma(i-j-1)} \int_{0}^{\infty} z_{i+1} e^{-N z_{i+1}} \int_{0}^{z_{i+1}} (z_{i+1} - z_{i})^{2H}   \nonumber \\
& \dfrac{\Gamma(i-j-1) \Gamma(j+2)}{\Gamma(i+1)} z_{i}^{i} dz_{i} dz_{i+1}  \nonumber \\
&=\dfrac{\left( 2^{2H-1}+1\right)^2}{N^{2}} \sum_{i,j=1, \dots, N; (i-j) \geq 2} \dfrac{N^{i+1}(j+1) \Gamma(2H+1) }{\Gamma(2H+i+2)}  \int_{0}^{\infty} e^{-N z_{i+1}}  z_{i+1}^{2H+i+2}  dz_{i+1}  \nonumber \\
&= \dfrac{\left( 2^{2H-1}+1\right)^2 \Gamma(2H+1)}{N^{2}} \sum_{i,j=1, \dots, N; (i-j) \geq 2} \dfrac{N^{i+1} (j+1)}{\Gamma(2H+i+2)} \dfrac{\Gamma(2H+i+3)}{N^{2H+i+3}} \nonumber  \\
&= \dfrac{\left( 2^{2H-1}+1\right)^2 \Gamma(2H+1)}{N^{2H+4}} \sum_{i,j=1, \dots, N; (i-j) \geq 2} (j+1) (2H+i+2)  \leq \dfrac{C_{10}(H)}{N^{2H}}.
\end{align}
To estimate $A_{N}^{(3,2)}$ and $A_{N}^{(3,3)}$ we proceed as  in the case of  $A_{N}^{(3,1)}$ obtaining 
\begin{equation}\label{A3-2-2}
A_{N}^{(3,2)} \le \frac{C_{11}}{N^{2H}}\quad \mbox{and} \quad  A_{N}^{(3,3)} \le \frac{C_{12}}{N^{2H}}.
\end{equation}
Finally combining (\ref{tauA22-2}), (\ref{tauA211-}), (\ref{tauA212}),  (\ref{A3-1-1}) and (\ref{A3-2-2})  we obtain the desired result.

\end{proof}

\section{Simulation Study}
In this section, we present a Monte Carlo simulation study to assess the finite sample properties for the least squares estimator in the linear regression model driven by a fractional Brownian motion evaluated at random times defined by equations (\ref{io}) and (\ref{rp}). 

First, we  will work with the model observed in equally spaced times, defined by Equation (\ref{modelo}). We simulate $M= 1000$ replicates of the model with $a=1$ and  different values of $N$ from $N=1$ to $N=300$, and we take the average of the  estimators  $\hat{a}_N$ for each $N$. 
We consider three values of Hurst parameter $H = 0.55;  0.75$ and $0.95$.  The data were simulated according to
$$Y_{t_i+1}=a \, t_{i+1}+\Delta B^H_{t_i+1},$$
\noindent where $\Delta B^H_{t_i+1}= B^H_{t_i+1}- B^H_{t_i}$ denotes the increments of the fractional Brownian motion and $t_{i+1}=\frac{i+1}{N}$.
We illustrate the convergence speed of the parameter estimations $\hat{a}_N$ with graphical analysis. In Figure (\ref{convdet})  we can see that for different values of $H$,  the estimation of the parameter reaches in very few iterations.

\begin{figure}[h!]
\centering
\includegraphics[scale=0.7]{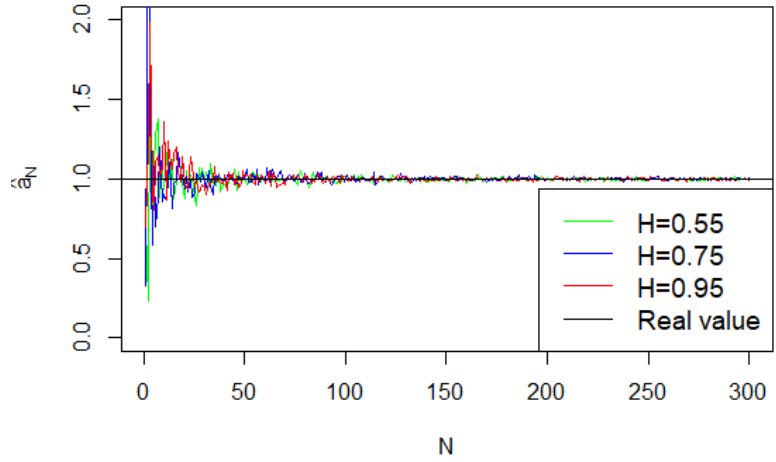}
\caption{Convergence of $\hat{a}_N$ for different values of $H$}\label{convdet}
\end{figure}

We consider now, the estimation of parameter in the case of the model observed at sampling random times. We take in this case the value $a=2$. Tables (\ref{tabla-io1}) and  (\ref{tabla-rp1}) present the simulation results for estimation in case 1: Jittered sampling (according to Equation (\ref{io})) and case 2: Renewal process (see Equation (\ref{rp})). Performance statistics presented are the mean, standard deviation (SD) and Kurtosis. The Kurtosis is defined as the difference between kurtosis of a Gaussian distribution and of the simulated process. 
For all $H$ values, there is a small bias with respect to the true parameter value. For $H=0.95$ there is a small standard deviation, which is expected since, in the context of long-range dependence processes, it is quite common for the process to be less noisy.
\newpage
\begin{table}[h!]
\centering
\begin{tabular}{cccc}
\hline \hline
$a=2$ & $H=0.55$ & $H=0.75$ & $H=0.95$ \\ \hline
Mean  & 1.998539 & 2.000285 & 1.999902 \\
SD    & 0.09950515 & 0.03627706 & 0.01497767 \\ 
Kurtosis & 0.518058 & 1.154266 &  0.817443\\ \hline \hline
\end{tabular}
\caption{Jittered sampling: Mean, standard deviation and kurtosis for different values of $H$ and $a = 1$} \label{tabla-io1}
\end{table}
\begin{table}[h!]
\centering
\begin{tabular}{cccc}
\hline \hline
$a=2$ & $H=0.55$ & $H=0.75$ & $H=0.95$ \\ \hline
Mean  & 2.006504 & 2.014654 & 2.002542 \\
SD    & 0.2197203 & 0.2208552 & 0.219923 \\ 
Kurtosis & 0.437066 & 0.078133 &  0.644303\\ \hline \hline
\end{tabular}
\caption{Renewal process: Mean, standard deviation and kurtosis for different values of $H$ and $a = 1$} \label{tabla-rp1}
\end{table}

As in the case of non-random times, we illustrate the convergence speed of the parameter estimations $\hat{a}_N$ with a graphical analysis. In Figure (\ref{RT-speed}) we can see that for different values of $H$. It takes averagely no more then $50$ iterations to reach complete convergence.
\begin{figure}[h!]
\begin{subfigure}[b]{0.3\textwidth}
        \includegraphics[width=4.6cm, angle=270]{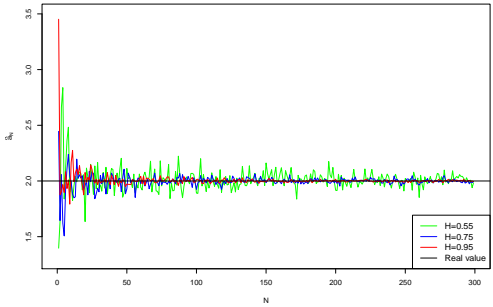}
        \caption{Convergence of $\hat{a}_N$ for different values of $H$ in Jittered sampling}
   \end{subfigure}
     ~   ~   ~   ~   ~   ~   ~   ~   ~   ~  ~   ~    
 \begin{subfigure}[b]{0.3\textwidth}
        \includegraphics[width=4.6cm, angle=270]{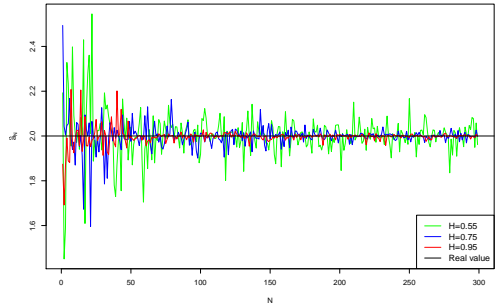}
        \caption{Convergence of $\hat{a}_N$ for different values of $H$ in Renewal process}
   \end{subfigure}
   \caption{Convergence of $\hat{a}_N$ for different values of $H$.}\label{RT-speed}
\end{figure}

Overall, the results in Tables  (\ref{tabla-io1}) and (\ref{tabla-rp1}), for all $H$'s show a smaller bias for the parameter and a considerably small standard deviation. As expected, as $N$ increases, the bias in the estimates decrease and so does the Standard Deviation (and consequently the variance and the MSE), which is a reflection of the consistency of the estimator.

\newpage
\begin{figure}[h!]
    \centering
    \begin{subfigure}[b]{0.3\textwidth}
        \includegraphics[width=3.5cm, angle=270]{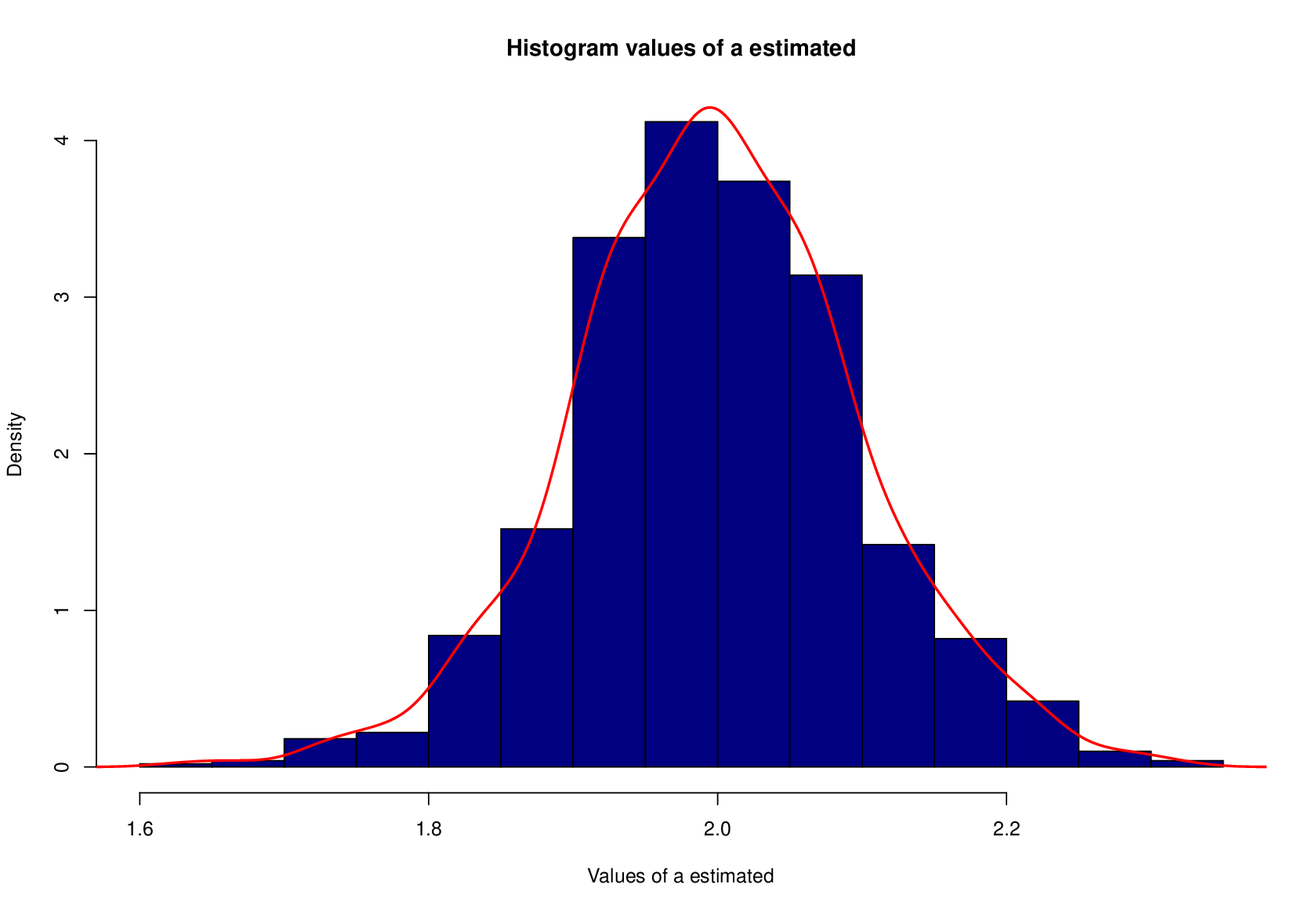}
        \caption{$H=0.55$}
   \end{subfigure}
    ~ 
    \begin{subfigure}[b]{0.3\textwidth}
        \includegraphics[width=3.5cm, angle=270]{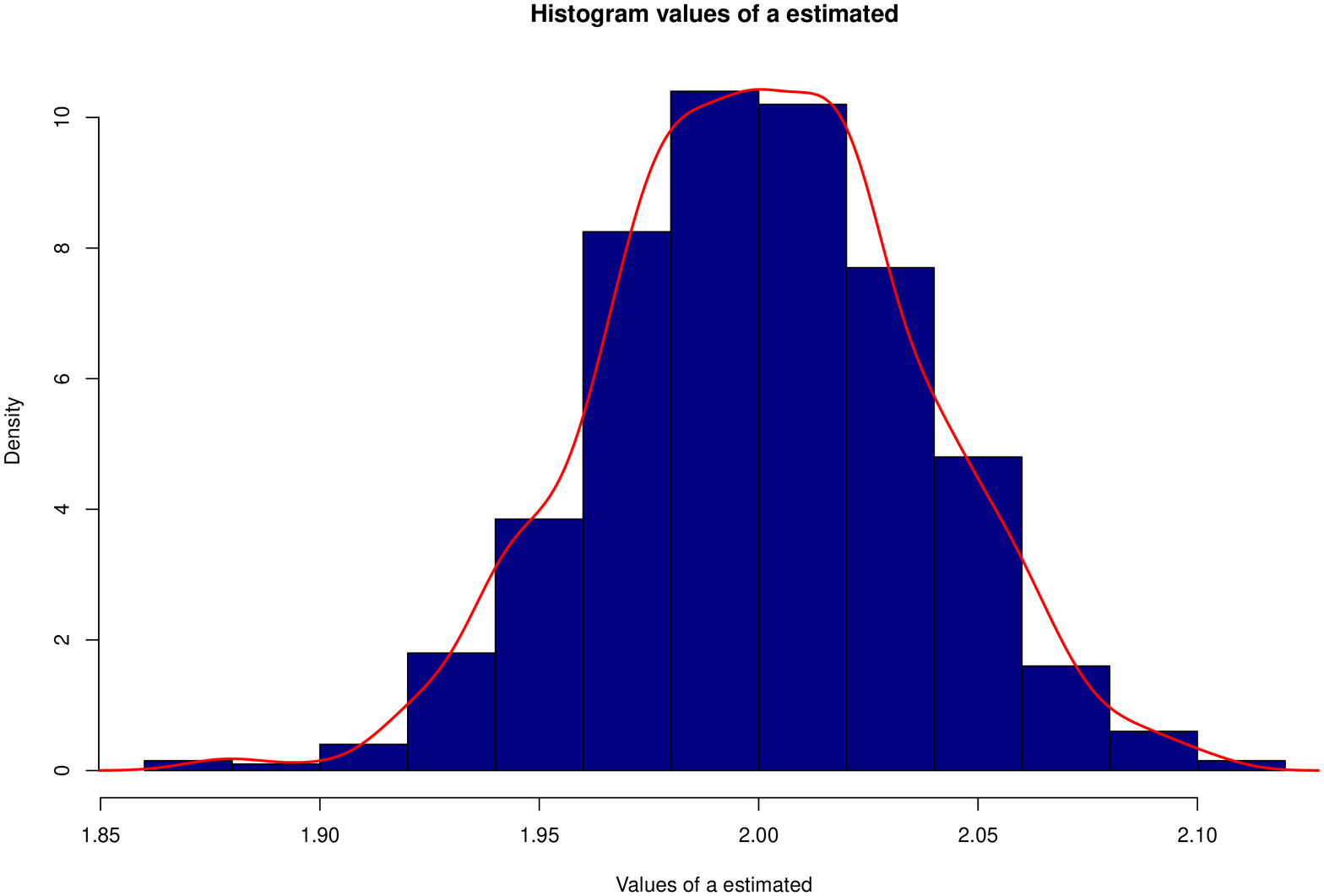}
        \caption{$H=0.75$}
 \end{subfigure}
    ~ 
    \begin{subfigure}[b]{0.3\textwidth}
        \includegraphics[width=3.5cm, angle=270]{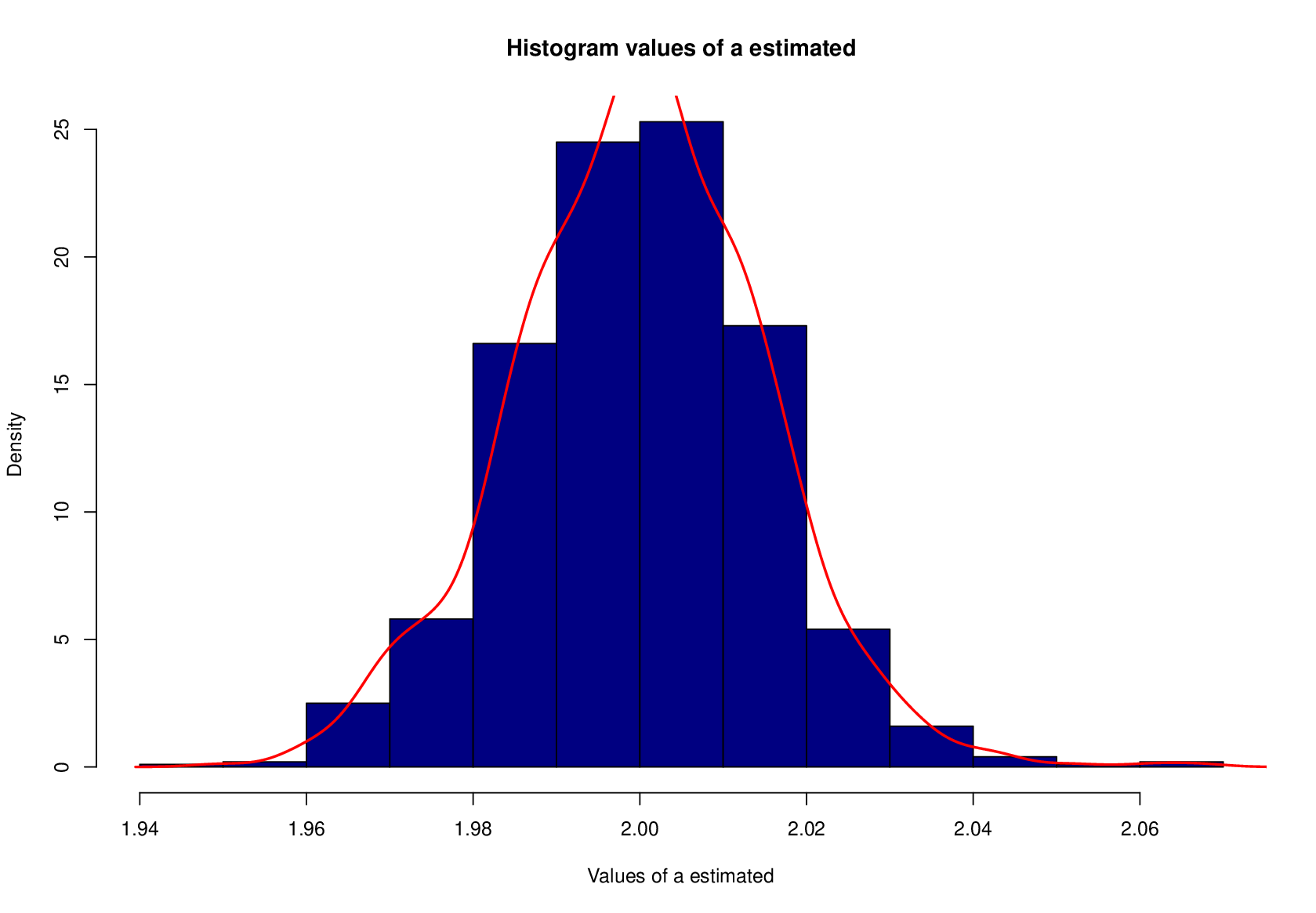}
        \caption{$H=0.95$}
 \end{subfigure}
    \caption{Jittered sampling: Histograms of values of $\hat{a}_N$ with different values of $H$}\label{io1}
\end{figure}
\begin{figure}[h!]
    \centering
    \begin{subfigure}[b]{0.3\textwidth}
        \includegraphics[width=3.5cm, angle=270]{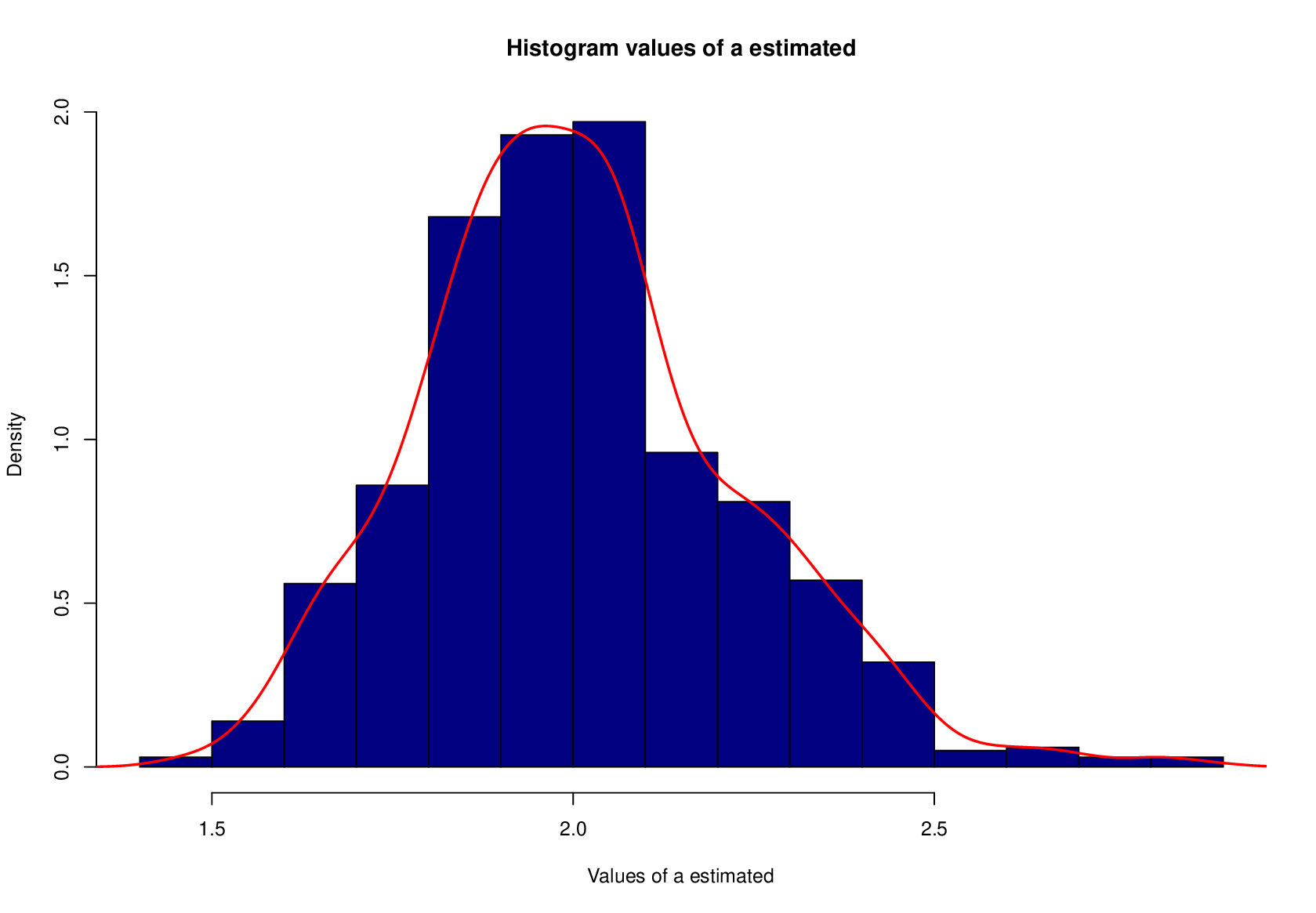}
        \caption{$H=0.55$}
  \end{subfigure}
    ~ 
    \begin{subfigure}[b]{0.3\textwidth}
        \includegraphics[width=3.5cm, angle=270]{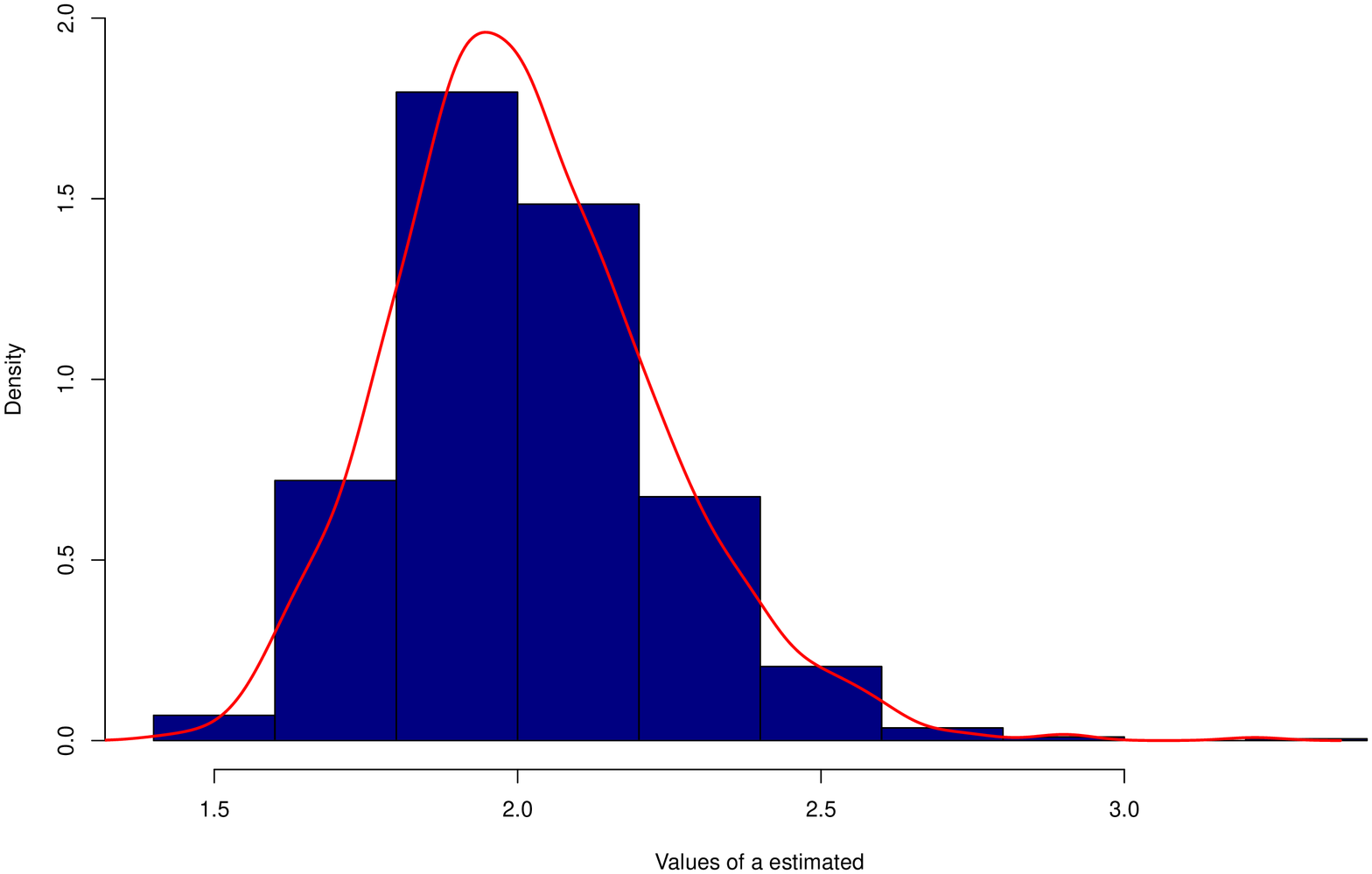}
        \caption{$H=0.75$}
  \end{subfigure}
    ~ 
    \begin{subfigure}[b]{0.3\textwidth}
        \includegraphics[width=3.5cm, angle=270]{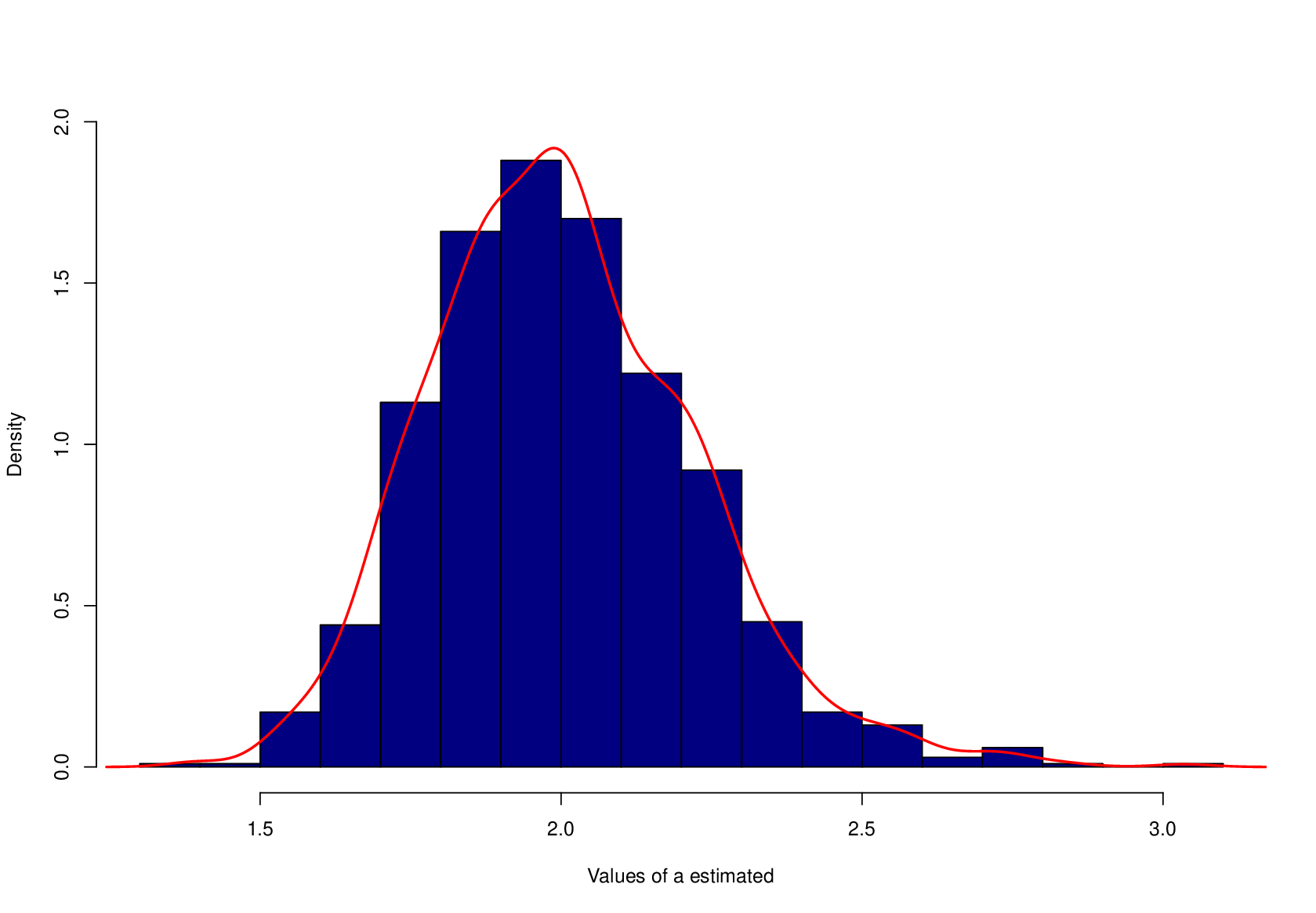}
        \caption{$H=0.95$}
  \end{subfigure}
    \caption{Renewal process: Histograms of values of $\hat{a}_N$ with different valor of $H$}\label{io10}
\end{figure}

Figures (\ref{io1})  and (\ref{io10}) show the frequency histograms (sampling distribution) of the 1000 values generated for different values of $H$ and the parameter $a=2$ follow Jittered sampling and Renewal processes respectively. 

A clear break occurs at $H = 0.75$ in the case of the renewal process, in which the histogram looks rather normal, and has asymmetric with a long right tail and no left tail, and a strong pointedness.

In conclusion, we can see that the estimation for  parameter $a$ is very precise and constant for all values of $H$ studied here, and for two different developing situations sampling schemes. Furthermore, in all cases and sampling, the estimation is very accurate showing that our estimation procedure is a good alternative to estimate parameters in a linear regression model with random times and long memory noise. 

\newpage

\section{Appendix: Joint distributions for renewal process  sampling and  estimates.}\label{A1}

In this section we present the joint distribution associated to the sequence of random variables  $\{ t_i, i=1, \ldots , N\}$ and $\{\tau_i,  i=1, \ldots , N\}$, where
\begin{align*}
\mbox{For} \quad  1 \leq i \leq N; \; t_{i} & \sim Gamma(1,N)\;  \text{are i.i.d. random variables, } \\ \mbox{For} \quad  1 \leq i \leq N; \; \tau_{i} &= \sum_{j=1}^{i} t_{j}  \sim Gamma(i,N) \; \text{and} \\
\mbox{For} \quad  1 \leq i \leq N-1; \; \tau_{i+1} &= \tau_{i} + t_{i+1}. \;\;  \text{Here $\tau_{i}$ is independent of $t_{i+1}$}.
\end{align*}

\begin{table}[h!]
\centering
\begin{tabular}{|c|c|c|}
\hline
Joint distribution & Probability Density Function & Support \\\hline
$f_{\tau_{i} , t_{i+1}}({a, b})$ & $\dfrac{N^{i+1}}{\Gamma(i)} a^{i-1} e^{-N (a + b)}$ & $0 \leq a < \infty$ \\
                                &                                                                                                          & $0 \leq b < \infty$ \\ \hline
$f_{\tau_{i} , \tau_{i+1}}({a, b}) $ & $\dfrac{N^{i+1}}{\Gamma(i)} a^{i-1}  e^{-N b}$ & $0 \leq a \leq b$  \\
                                  &                                                                                          & $0 \leq b < \infty$ \\ \hline
$f_{\tau_{i-1} , t_{i}, t_{i+1}}({a, b, c})$ & $\dfrac{N^{i+1}}{\Gamma(i-1)} a^{i-2} e^{-N (a + b + c)}$ & $0 \leq a < \infty $ \\
                                                &                                                                                                                                     & $0 \leq b < \infty$ \\ 
                                                &                                                                                                                                     &  $0 \leq c < \infty$ \\ \hline
$f_{\tau_{i-1} , \tau_{i} , \tau_{i+1}}({a, b, c})$ & $\dfrac{N^{i+1}}{\Gamma(i-1)} a^{i-2} e^{-N c}$ & $0 \leq a \leq b$ \\
                                              &                                                                                                   & $0 \leq b \leq c$ \\
                                              &                                                                                                   & $0 \leq c < \infty$ \\ \hline
$f_{\tau_{j} , \tau_{j+1} , \tau_{i} , \tau_{i+1}}({a, b, c,d})$ & $\dfrac{N^{i+1}}{\Gamma(j) \Gamma(i-j-1)} a^{j-1} \left(c - b  \right)^{i-j-2} e^{-N d}$ & $0 \leq a \leq b$\\ 
                                                           &                                                                                      & $0 \leq b \leq c$ \\
                                                           &                                                                                      & $0 \leq c \leq d$ \\
                                                           &                                                                                      &  $0 \leq d < \infty$ \\ \hline
\end{tabular}
\caption{Densities under Renewal Process} \label{densities}
\end{table}

Also we present estimates used to develop some of the computations that have been used throughout the article. 
\begin{equation} \label{inte1}
\int_{0}^{a} \left( a - x \right) ^{b} x^{c} dx = \dfrac{\Gamma(b+1) \Gamma(c+1)}{\Gamma(b+c+2)} a^{b+c+1}.
\end{equation}
\begin{equation} \label{inte2}
\int_0^\infty x^a e^{-b x} dx = b^{-a-1} \Gamma(a + 1) \quad \mbox{for} \quad Re(b)>0 \quad \mbox{and} \quad  Re(a)>-1
\end{equation}

\section{Appendix: Proof of Lemma \ref{lem1} }\label{AL}
\begin{proof}
In order to prove Lemma \ref{lem1}, we will first analyze the case when $\tau$ satisfies (\ref{io}), that is  we consider the case of the jittered sampling.  By  definition of $D_N$, (\ref{DN}), we have
\begin{eqnarray*}
D_{N} &=& \dfrac{1}{N} \sum_{i=0}^{N-1} \left( \dfrac{i+1}{N} + \nu_{i+1} \right)^{2} \\
&=& \dfrac{1}{N} \sum_{i=0}^{N-1} \dfrac{(i+1)^{2}}{N^{2}} + \dfrac{2}{N} \sum_{i=0}^{N-1} \dfrac{(i+1) \nu_{i+1}}{N} + \dfrac{1}{N} \sum_{i=0}^{N-1} \nu_{i+1}^{2} \\
&=& I_N^1 + I_N^2 + I_N^3.
\end{eqnarray*}
First  we have $I_N^{1} = \dfrac{1}{N^{3}} \sum_{i=0}^{N-1} (i+1)^{2} = \dfrac{2N^{3} + 3N^{2} + N}{6 N^{3}}.$
Then, 
\begin{equation}\label{I1N}
\displaystyle\lim_{N \to \infty} I_N^1 = \frac{1}{3}.
\end{equation}
Now,  $I_{2}^N$ can be written as follows
\begin{align*}
I_{2} &= \sum_{i=0}^{N-1} a_{i,N} \nu_{i+1},
\end{align*}
where $a_{i,N} = \frac{2(i+1)}{N^{2}}$. Since   $\max_{1 \leq i \leq N} \vert a_{i,N} \vert \leq  O \left( \frac{1}{N} \right)$; $\mathbb{E}(\nu_{i})=0$ and $\nu_{i} \leq \frac{1}{2N} $ for all $i=0, \ldots, N-1$, we can apply   Theorem 5 in \cite{dae1987}, to obtain

\begin{align} \label{unif-i2}
I_{2} = \sum_{i=0}^{N-1} a_{i,N} \nu_{i+1} \xrightarrow[N \to \infty]{a.s.}   0.
\end{align}

For the  third term $I_N^3$,we take into account that $\nu_{i+1} \in [-1/2N , 1/2N]$ for all $i=0, \ldots , N-1$. Consequently 
\begin{equation}\label{unif-i3}
\dfrac{1}{N} \sum_{i=0}^{N-1} \nu_{i+1}^{2}   \leq  \dfrac{1}{N} \sum_{i=0}^{N-1} \dfrac{1}{4N^2} = \dfrac{1}{4N^2} \xrightarrow[N \to \infty]{a.s.}   0.
\end{equation}
Finally, by (\ref{I1N}), (\ref{unif-i2}) and (\ref{unif-i3})  the result is achieved. \\
Let us consider $D_N$ with the sampling random times  as in  (\ref{rp}). We have 
\begin{align*}
D_{N} = \dfrac{1}{N} \sum_{i=0}^{N-1} \tau_{i+1}^{2} &= \dfrac{1}{N} \sum_{k=1}^{N} \left( \sum_{i=1}^{k} t_{i} \right)^{2} = \dfrac{1}{N} \sum_{k=1}^{N} \sum_{i=1}^{k} \sum_{j=1}^{k} t_{i} t_{j} \\
&= \dfrac{1}{N} \sum_{i=1}^{N} \sum_{j=1}^{N} \left( N - (i \vee j) + 1 \right) t_{i} t_{j} \\
&= \dfrac{1}{N} \sum_{i=1}^{N} \left( N-i+1  \right) t_{i}^{2} + \dfrac{1}{N} \sum_{1 \leq i \neq j \leq N} \left( N - (i \vee j ) + 1 \right) t_{i} t_{j}.
\end{align*}
It is important to recall that the random times are not independent, so the previously used techniques are not directly applicable. However, as shown before $D_N$, can be written as a quadratic form depending on the increments $t$, which are independent. We define the following  variables 
\begin{align}
R_{N} &= \mathbb{E} \left[ D_{N} \right] \label{rn}, \\
T_{N} &= \dfrac{1}{N} \sum_{i=1}^{N} \left( N-i+1 \right) \left( t_{i}^{2} - \mathbb{E} \left[ t_{i}^{2} \right] \right) \label{tn}, \\
Q_{N} &= \dfrac{1}{N} \sum_{1 \leq i \neq j \leq N} \left( t_{i} \mathbb{E} \left[ t_{j} \right] + \mathbb{E} \left[ t_{i} \right] t_{j} - 2 \mathbb{E} \left[ t_{i} \right] \mathbb{E} \left[ t_{j} \right] \right) \left( N - (i \vee j)+1 \right) \label{qn}, \\
U_{N} &= \dfrac{1}{N} \sum_{1 \leq i \neq j \leq N} \left( t_{i} t_{j} - t_{i} \mathbb{E} \left[ t_{j} \right] - t_{j} \mathbb{E} \left[ t_{i} \right] + \mathbb{E} \left[ t_{i} \right] \mathbb{E} \left[ t_{j} \right] \right) \left( N - (i \vee j )+1 \right) \label{un}.
\end{align}
Then, $D_N$ can be decompose as follows: $
D_{N} = R_{N} + T_{N} + Q_{N} + U_{N}. $
Now, we will show that $R_{N}$ converges to $1/3$, $T_{N}$, $Q_{N}$ and $U_{N}$ converges to $0$ as $N$ goes to infinity.\\

\noindent \textbf{{Convergence of $R_{N}$ to $1/3$}}. Let us  recall that  for all $i=1, \ldots, N$, we have that  $\mathbb{E} \left[ t_{i} \right] = \frac{1}{N}$ and $\mathbb{E} \left[ t_{i}^{2} \right] = \frac{2}{N^{2}}$. Then
\begin{align*} 
R_{N} = \mathbb{E} \left[ D_{N} \right] &= \dfrac{1}{N} \sum_{i=1}^{N} \sum_{j=1}^{N} \left( N- (i \vee j) + 1 \right) \mathbb{E} \left[ t_{i} t_{j} \right] \\
&= \dfrac{2}{N^{3}} \sum_{i=1}^{N} \left( N - i +1 \right) + \dfrac{2}{N^{3}} \sum_{1 \leq i < j \leq N} \left( N-j+1 \right) \\
&= \dfrac{2}{N^{3}} \sum_{i=1}^{N} i + \dfrac{2}{N^{3}} \sum_{j=2}^{N} \sum_{i=1}^{j-1} \left( N-j+1 \right) \\
&= \dfrac{2}{N^{3}} \dfrac{N(N+1)}{2} + \dfrac{2}{N^{3}} \sum_{j=2}^{N} \left( N-j+1 \right) (j-1) \\ 
&= \dfrac{N(N+1)}{N^{3}} + \dfrac{N(N-1)(N+1)}{3N^{3}},
\end{align*}
which converges to $1/3$ a.s. as $N \to \infty$.\\

\noindent \textbf{{Almost sure convergence of $T_{N}$ to $0$}}. 
We will prove the  convergence to zero in $L^2$. Then by applying Borel Cantelli lemma we will get  the a.s. convergence to 0. 
\begin{align*}
\mathbb{E} \left[  T_{N}^{2} \right] &= \dfrac{1}{N^{2}} \sum_{i=1}^{N} \sum_{j=1}^{N} \left( N-i+1 \right) \left( N-j+1 \right) \mathbb{E} \left[ \left( t_{i}^{2} - \mathbb{E} \left[ t_{i}^{2} \right] \right) \left( t_{j}^{2} - \mathbb{E} \left[ t_{j}^{2} \right] \right)  \right] \\
&= \dfrac{1}{N^{2}} \sum_{i=1}^{N} \left( N-i+1 \right)^{2} \mathbb{E} \left[ \left( t_{i}^{2} - \frac{2}{N^{2}} \right)^{2} \right] \\
&+ \dfrac{1}{N^{2}} \sum_{1 \leq i \neq j \leq N} \left( N-i+1 \right) \left( N-j+1 \right) \mathbb{E} \left[ \left( an t_{i}^{2} - \frac{2}{N^{2}} \right)  \left( t_{j}^{2} - \frac{2}{N^{2}} \right) \right] \\
&= \dfrac{1}{N^{2}} \sum_{i=1}^{N} \left( N-i+1 \right)^{2} \left( \dfrac{3 !}{N^{4}} - \left( \frac{2}{N^{2}} \right)^{2} \right) \\
&= \dfrac{2}{N^{6}} \sum_{i=1}^{N} i^{2} 
= \dfrac{2}{6} \dfrac{N(N+1)(2N+1)}{N^{6}},
\end{align*}
which converges to zero as $N$ tends  to $\infty$.
Using Borel Cantelli Lemma, we get 
\begin{align*}
\sum_{N=1}^{\infty} \mathbb{P} \left( |T_{N}| > \epsilon \right) & \leq \dfrac{1}{\epsilon^{2}} \sum_{N=1}^{\infty}  \mathbb{E} \left[ |T_{N}|^{2} \right]  
\dfrac{1}{N^{3}} < \infty.
\end{align*}
Then, $T_{N} \xrightarrow[N \to \infty]{a.s.} 0 $.\\

\noindent \textbf{{Almost sure convergence of $Q_{N}$ to $0$}}. \\
Given $Q_N$ in (\ref{qn})  we can write  as a weighted  sum of i.i.d. random variables as follows

\begin{align*}
Q_{N} &= \dfrac{1}{N} \sum_{1 \leq i \neq j \leq N} \left( t_{i} \frac{1}{N} + t_j \frac{1}{N} - \frac{2}{N^{2}} \right) \left(  N - (i \vee j) + 1 \right) \\
&= \dfrac{2}{N} \sum_{1 \leq i < j \leq N} \left( t_{i} \frac{1}{N} + \frac{1}{N} t_{j} - \frac{2}{N^{2}} \right)  \left( N-j+1 \right) \\
&= \dfrac{2}{N} \sum_{1 \leq i < j \leq N} \left( \frac{t_{i}}{N} - \frac{1}{N^{2}} \right) \left( N-j+1 \right) + \dfrac{2}{N} \sum_{1 \leq i < j \leq N} \left( \frac{t_{j}}{N} - \frac{1}{N^{2}} \right) \left( N-j+1 \right) \\
&= \dfrac{2}{N} \sum_{i=1}^{N-1} \sum_{j=i+1}^{N} \left( \frac{t_{i}}{N} - \frac{1}{N^{2}} \right) \left( N-j+1 \right) + \dfrac{2}{N} \sum_{j=2}^{N} \sum_{i=1}^{j-1} \left( \frac{t_{j}}{N} - \frac{1}{N^{2}} \right) \left( N-j+1 \right) \\
&= \dfrac{2}{N} \sum_{i=1}^{N-1} \left( \frac{t_{i}}{N} - \frac{1}{N^{2}} \right) \left( \sum_{j=1}^{N-i} j \right) + \dfrac{2}{N} \sum_{j=2}^{N} \left( \frac{t_{j}}{N} - \frac{1}{N^{2}} \right) \left( N-j+1 \right) (j-1) \\
&= \dfrac{2}{N} \sum_{i=1}^{N-1} \left( \frac{t_{i}}{N} - \frac{1}{N^{2}} \right) \dfrac{(N-i)(N-i+1)}{2} + \dfrac{2}{N} \sum_{j=2}^{N} \left( \frac{t_{j}}{N} - \frac{1}{N^{2}} \right) \left( N-j+1 \right) (j-1) \\
&= \dfrac{2}{N} \sum_{i=1}^{N} \left( \frac{t_{i}}{N} - \frac{1}{N^{2}} \right) \left[ \dfrac{(N-i)(N-i+1)}{2} + (N-i+1)(i-1) \right] \\
&= \sum_{i=1}^{N} \left( N t_{i} - 1 \right) \dfrac{(N-i+1)(N+i-2)}{N^{3}}.
\end{align*}
Note that for all $i=1, \ldots , N$  the random variable  $X_{i} = N t_{i}$, has a exponential distribution with parameter $1$. 
%
i.e. $X_{i} \sim exp(\lambda = 1)$. Besides,  $X_i, 1 \le i \le N$ are i.i.d. random variables. Then $Q_N$ can be written as a weighted sum of i.i.d. random variables as
\begin{align*}
Q_{N} &= \sum_{i=1}^{N} \left( X_{i} - \mathbb{E} \left[ X_{i} \right] \right) \dfrac{(N-i+1)(N+i-2)}{N^{3}} 
= \sum_{i=1}^{N} \left( X_{i} - \mathbb{E} \left[ X_{i} \right] \right) a_{N,i},
\end{align*}
where $a_{N,i}$ are such that $\max_{1 \leq i \leq N} \vert a_{N,i} \vert \leq O \left( \frac{1}{N} \right)$. Then, by \cite{dae1987} (Theorem 5), we conclude that
\begin{align*}
Q_{N} = \sum_{i=1}^{N} a_{N,i} \left( X_{i} - \mathbb{E} \left[ X_{i} \right] \right) \xrightarrow[N \to \infty]{a.s.} 0.
\end{align*}
\textbf{{Almost sure convergence of $U_{N}$ to $0$}}. Let us consider 
\begin{equation*}
U \left( t_{i}, t_{j} \right) = t_{i} t_{j} - \frac{t_{i} }{N} - \frac{t_{j} }{N} + \frac{1}{N^{2}}. 
\end{equation*}
It follows that $\mathbb{E} \left[ U \left( t_{i}, t_{j} \right)  \right] = 0$ and $\mathbb{E} \left[ U \left( t_{i}, t_{j} \right) U \left( t_{k}, t_{l} \right)  \right] = 0$ for $(i,j) \neq (k,l)$. Therefore  

\begin{align*}
Var \left( U_{N} \right) &= \mathbb{E} \left[ U_{N}^{2} \right] = \dfrac{1}{N^{2}} \sum_{1 \leq i \neq j \leq N} \mathbb{E} \left[ \left( t_{i} t_{j} - \frac{t_{i}}{N} - \frac{t_{j}}{N} + \frac{1}{N^{2}} \right)^{2} \right] \left( N - (i \vee j) + 1 \right) ^{2} \\
&= \dfrac{1}{N^{2}} \sum_{1 \leq i \neq j \leq N} \left[ \mathbb{E} \left[ t_{i}^{2} t_{j}^{2} \right] + \frac{\mathbb{E} \left[ t_{i}^{2} \right] }{N^{2}} + \frac{\mathbb{E} \left[ t_{j}^{2} \right]}{N^{2}} + \frac{1}{N^{4}} 
- 2 \frac{\mathbb{E} \left[ t_{i}^{2} t_{j} \right]}{N} - 2 \frac{\mathbb{E} \left[ t_{i} t_{j}^{2} \right]}{N} \right. \\
&+ 2 \frac{\mathbb{E} \left[ t_{i} t_{j} \right]}{N^{2}} + \left. 2 \frac{\mathbb{E} \left[ t_{i} t_{j} \right]}{N^{2}} - 2 \frac{\mathbb{E} \left[ t_{i}  \right]}{N} - 2 \frac{\mathbb{E} \left[t_{j} \right]}{N^{2}} \right] \left( N -(i \vee j) +1 \right)^{2} \\
&= \dfrac{1}{N^{2}} \sum_{1 \leq i \neq j \leq N} \left[ \dfrac{13}{N^{4}} -  \dfrac{12}{N^{4}} \right] \left( N -(i \vee j) +1 \right)^{2} = \dfrac{1}{N^{6}} \sum_{1 \leq i \neq j \leq N} \left( N -(i \vee j) +1 \right)^{2}.
\end{align*}
which is of order $O \left( \frac{1}{N^{2}} \right).$
Using Borell-Cantelli, we have 
\begin{equation*}
\sum_{N=1}^{\infty} \mathbb{P} \left( \vert U_{N} \vert > \epsilon  \right) \leq \frac{1}{\epsilon^{2}} \sum_{N=1}^{\infty} \mathbb{E}\left( \vert U_{N} \vert^2 \right) \leq  \frac{1}{\epsilon^{2}} \sum_{N=1}^{\infty} \frac{1}{N^2} < \infty.
\end{equation*}
Therefore  $U_{N} \xrightarrow[N \to \infty]{a.s.} 0.$
Finally, the almost sure convergence of $D_{N}$ to $1/3$ is a consequence of the almost sure convergence of $T_{N}$, $Q_{N}$ and $U_{N}$ to $0$, and the convergence of $R_{N}$ to $1/3$.

\end{proof}

\section*{Acknowledgements}
This  research was partially supported by Project ECOS - CONICYT C15E05, REDES 150038 and MATHAMSUD 18-MATH-03 SaSMoTiDep Project. H\'ector Araya was partially supported by Proyecto FONDECYT Post-Doctorado 3190465, Natalia Bahamonde was partially supported by FONDECYT Grant 1160527, Tania Roa was partially supported by Beca CONICYT-PFCHA / Doctorado Nacional / 2018-21180298, Soledad Torres was partially supported by FONDECYT Grant 1171335.

\bibliographystyle{plain} 

\begin{thebibliography}{10}

\bibitem{PRV} Anne Philippe, Caroline Robet, Marie-Claude Viano.
Random discretization of stationary continuous time processes.
\newblock {\em hal-01944290}, 2018.

\bibitem{bardet}
Jean-Marc Bardet and Pierre~R. Bertrand.
\newblock A non-parametric estimator of the spectral density of a
  continuous-time Gaussian process observed at random times.
\newblock {\em Scandinavian Journal of Statistics}, 37(3):458--476, 2010.

\bibitem{bell1981}
David~R. Bellhouse.
\newblock Area estimation by point-counting techniques.
\newblock {\em Biometrics}, pages 303--312, 1981.

\bibitem{chang2014}
Chin-Chih Chang.
\newblock Optimum preventive maintenance policies for systems subject to random
  working times, replacement, and minimal repair.
\newblock {\em Computers \& Industrial Engineering}, 67:185--194, 2014.

\bibitem{dae1987}
Bong Dae~Choi and Soo Hak~Sung.
\newblock Almost sure convergence theorems of weighted sums of random
  variables.
\newblock {\em Stochastic Analysis and Applications}, 5(4):365--377, 1987.

\bibitem{durrett}
Richard Durrett.
\newblock {\em Essentials of Stochastic Processes}.
\newblock Springer Texts in Statistics, 2012.

\bibitem{khan2017}
Muzibur Khan.
\newblock Performance testing of computed radiography system and imaging
  plates.
\newblock In {\em ASNT Annual Conference 2017}, pages 83--92, 2017.

\bibitem{krune2016}
Edgar Krune, Benjamin Krueger, Lars Zimmermann, Karsten Voigt, and Klaus
  Petermann.
\newblock Comparison of the jitter performance of different photonic sampling
  techniques.
\newblock {\em Journal of Lightwave Technology}, 34(4):1360--1367, 2016.

\bibitem{mandelbrotfractional}
Benoit~B. Mandelbrot and John~W.... Van~Ness.
\newblock Fractional brownian motions, fractional noises and applications.
\newblock {\em SIAM review}, 10(4):422--437, 1968.

\bibitem{masry1}
Elias Masry.
\newblock Probability density estimation from sampled data.
\newblock {\em IEEE Trans. Inform. Theory}, 29(5):696--709, 1983.

\bibitem{max2014}
W.~Max-Moerbeck, JL.~Richards, T.~Hovatta, V.~Pavlidou, TJ.~Pearson, and ACS.
  Readhead.
\newblock A method for the estimation of the significance of cross-correlations
  in unevenly sampled red-noise time series.
\newblock {\em Monthly Notices of the Royal Astronomical Society},
  445(1):437--459, 2014.

\bibitem{nieto2015}
Luis~E. Nieto-Barajas and Tapen Sinha.
\newblock Bayesian interpolation of unequally spaced time series.
\newblock {\em Stochastic environmental research and risk assessment},
  29(2):577--587, 2015.

\bibitem{olaf2016}
Krist{\'\i}n~Bj{\"o}rg {\'O}lafsd{\'o}ttir, Michael Schulz, and Manfred
  Mudelsee.
\newblock Redfit-x: Cross-spectral analysis of unevenly spaced paleoclimate
  time series.
\newblock {\em Computers \& Geosciences}, 91:11--18, 2016.

\bibitem{subr2014}
Kartic Subr, Derek Nowrouzezahrai, Wojciech Jarosz, Jan Kautz, and Kenny
  Mitchell.
\newblock Error analysis of estimators that use combinations of stochastic
  sampling strategies for direct illumination.
\newblock In {\em Computer Graphics Forum}, volume~33, pages 93--102. Wiley
  Online Library, 2014.

\bibitem{vilar}
Jose~A. Vilar.
\newblock Kernel estimation of the regression function with random sampling
  times.
\newblock {\em Test}, 4(1):137--178, 1995.

\bibitem{vilar2000}
Jose~A. Vilar and Juan~M. Vilar.
\newblock Finite sample performance of density estimators from unequally spaced
  data.
\newblock {\em Statistics and Probability Letters}, 50:63--73, 2000.

\bibitem{zhao2014}
Xufeng Zhao, Mingchih Chen, and Toshio Nakagawa.
\newblock Optimal time and random inspection policies for computer systems.
\newblock {\em Appl. Math}, 8(1L):413--417, 2014.

\end{thebibliography}

\end{document}